\documentclass[12pt]{amsart}
\usepackage{amssymb,amsthm}
\usepackage{hyperref} 
\usepackage{dsfont}	
\usepackage{color}	

\usepackage[all]{xy}
\usepackage{mathrsfs}

\newtheorem{theorem}{Theorem}[section]
\newtheorem*{theorem*}{Theorem}
\newtheorem{proposition}[theorem]{Proposition}
\newtheorem{coro}[theorem]{Corollary}
\newtheorem{lemma}[theorem]{Lemma}
\newtheorem{rem}[theorem]{Remark}
\newtheorem{definition}[theorem]{Definition}
\newtheorem{assum}{Assumption}[section]

\newtheorem{assum-T}{Assumptions on $T$}
\newtheorem*{assum-L}{Assumptions on $L$}

\theoremstyle{plain}

\renewcommand{\epsilon}{\varepsilon}
\newcommand{\eps}{\epsilon}

\newcommand{\dive}{\textrm{div}}

\newcommand{\calD}{\mathcal{D}}
\newcommand{\calS}{\mathcal{S}}

\newcommand{\mcD}{{\mathcal D}}
\newcommand{\mcQ}{{\mathcal Q}}
\newcommand{\mcS}{{\mathcal S}}
\newcommand{\ch}{{\mathrm ch}}

\DeclareMathOperator*{\essinf}{ess\,inf}
\DeclareMathOperator*{\esssup}{ess\,sup}

\newcommand\C{\mathbb{C}}
\newcommand\N{\mathbb{N}}
\newcommand\R{\mathbb{R}}

\newcommand\Z{\mathbb{Z}}

\newcommand{\calM}{\ensuremath{\mathcal{M}}}

\DeclareMathOperator{\supp}{supp}
\DeclareMathOperator{\loc}{loc}
\DeclareMathOperator{\diam}{diam}
\DeclareMathOperator{\ball}{ball}
\newcommand{\Eins}{\ensuremath{\mathds{1}}}

\def\Xint#1{\mathchoice
   {\XXint\displaystyle\textstyle{#1}}%
   {\XXint\textstyle\scriptstyle{#1}}%
   {\XXint\scriptstyle\scriptscriptstyle{#1}}%
   {\XXint\scriptscriptstyle\scriptscriptstyle{#1}}%
   \!\int}
\def\XXint#1#2#3{{\setbox0=\hbox{$#1{#2#3}{\int}$}
     \vcenter{\hbox{$#2#3$}}\kern-.5\wd0}}

\def\aver#1{\Xint-_{#1}}

\makeatletter
\let\original@addcontentsline\addcontentsline
\newcommand{\dummy@addcontentsline}[3]{}
\newcommand{\DeactivateToc}{\let\addcontentsline\dummy@addcontentsline}
\newcommand{\ActivateToc}{\let\addcontentsline\original@addcontentsline}
\makeatother

\newcommand\restr[2]{{% we make the whole thing an ordinary symbol
  \left.\kern-\nulldelimiterspace % automatically resize the bar with \right
  #1 % the function
  \vphantom{\big|} % pretend it's a little taller at normal size
  \right|_{#2} % this is the delimiter
  }}

\numberwithin{theorem}{section}
\numberwithin{equation}{section}

\title[Sharp weighted estimates]{Sharp weighted norm estimates beyond Calder\'on-Zygmund theory}

\author[Fr\'ed\'eric Bernicot, Dorothee Frey, Stefanie Petermichl]{Fr\'ed\'eric Bernicot, Dorothee Frey, Stefanie Petermichl}

\address{Fr\'ed\'eric Bernicot, CNRS - Universit\'e de Nantes, Laboratoire Jean Leray, 2 rue de la Houssini\`ere, 44322 Nantes cedex 3, France\\}
\email{frederic.bernicot@univ-nantes.fr}

\address{Dorothee Frey, Delft Institute of Applied Mathematics, Delft University of Technology, P.O. Box 5031, 2600 GA Delft, The Netherlands\\}
\email{d.frey@tudelft.nl}

\address{Stefanie Petermichl, IMT Universit\'e Paul Sabatier, 118 Route de Narbonne, 31062 Toulouse cedex 9, France\\}
\email{stefanie.petermichl@math.univ-toulouse.fr}

\thanks{This project is partly supported by ANR projects AFoMEN no. 2011-JS01-001-01 and HAB no. ANR-12-BS01-0013, as well as the ERC project FAnFArE no. 637510. The third author is a member of IUF}

\date{\today}

\begin{document}

\begin{abstract}  
We dominate non-integral singular operators by adapted sparse operators and derive optimal norm estimates in weighted spaces. Our assumptions on the operators are minimal and our result applies to an array of situations, whose prototype are Riesz transforms / multipliers or paraproducts associated with a second order elliptic operator. It also applies to such operators whose unweighted continuity is restricted to Lebesgue spaces with certain ranges of exponents $(p_0,q_0)$ where $1\le p_0<2<q_0\le \infty$. The norm estimates obtained are powers $\alpha$ of the characteristic used by Auscher and Martell. The critical exponent in this case is $\mathfrak{p}=1+\frac{p_0}{q'_0}$. We prove $\alpha=\frac1{p-p_0}$ when $p_0<p\le \mathfrak{p}$ and $\alpha= \frac{q_0-1}{q_0-p}$ when $\mathfrak{p}\le p<q_0$. 
In particular, we are able to obtain the sharp $A_2$ estimates for non-integral singular  operators which do not fit into the class of Calder\'on-Zygmund operators. These results are new even in the Euclidean space and are the first ones for operators whose kernel does not satisfy any regularity estimate.
\end{abstract}



\maketitle

\section{Introduction}

In the last ten years, it has been of great interest to obtain optimal operator norm estimates in Lebesgue spaces endowed with Muckenhoupt weights. One asks for the growth of the norm of certain operators, such as the Hilbert transform or the Hardy-Littlewood maximal function, with respect to a characteristic assigned to the weight. Originally, the main motivation for sharp estimates of this type came from certain important applications to partial differential equations. See for example  Fefferman-Kenig-Pipher \cite{FKP}, Astala-Iwaniec-Saksman \cite{AIS}. Indeed, a long standing regularity problem has been solved through the optimal weighted norm estimate of the Beurling-Ahlfors operator, a classical Calder\'on-Zygmund operator. See  Petermichl-Volberg \cite{PV}. Since then, the area has been developing rapidly.  Advances have greatly improved conceptual understanding of classical objects such as Calder\'on-Zygmund operators. The latter are now understood in several different ways, one of them being through pointwise control by so-called sparse operators, see most recently Lacey \cite{Lacey}, Lerner-Nazarov \cite{LN}.  We bring this circle of ideas to the wide range of non-integral singular operators, such as considered in Auscher-Martell \cite{AM}. Under minimal assumptions, we now demonstrate control by well chosen sparse operators and derive optimal norm estimates in weighted spaces. 

\medskip

From a historic standpoint, starting in 1973 when Hunt-Muckenhoupt-Wheeden  \cite{HMW} proved that in the Euclidean space, the Hilbert transform is bounded on $L^2_\omega$ if and only if the weight $\omega$ satisfies the so-called $A_2$ condition. Then the extension for $p\in(1,\infty)$ of the class $A_p$ for weights have been legitimate by the characterization of the Hardy-Littlewood maximal operator on $L^p_{\omega}$. These classes as well as the ``dual classes'' $RH_q$ (describing a Reverse H\"older property) are only defined in terms of volume of balls, so this entire theory has been extended to the doubling framework. Calder\'on-Zygmund operators have been proved to be bounded on $L^p_\omega$ if $\omega\in A_p$.
More recently, the so-called $A_p$-conjecture (which is now solved) was about the sharp dependence of this operator-norm with respect to the $A_p$ characteristic of the weight. This conjecture was solved by Petermichl in \cite{Petermichl} for the Hilbert transform and then by Hyt\"onen in \cite{Hytonen} for arbitrary Calder\'on-Zygmund operators. Recently Lerner (\cite{Lerner1,Lerner2,Lerner3}) has obtained an alternate proof of this result, by exploiting the notion of {\it local mean oscillation} in order to control the norm of a Calder\'on-Zygmund operators by the norm of some specific operators, called {\it sparse operators}. After that and very recently, Lacey in \cite{Lacey} and Lerner-Nazarov \cite{LN} gave another proof, which gets around the use of local oscillation through pointwise control. 

\medskip

Simultaneously, during the last years, people were also interested in weighted estimates for non-integral singular operators in a space of homogeneous type. Even on the Euclidean space, Riesz transforms $\nabla L^{-1/2}$ may be considered in several situations where we do not have pointwise regularity estimates of an integral kernel, for examples $L=-\textrm{div}(A\nabla)$ with bounded coefficients $A$, $L=-\Delta+V$ with some potential $V$. The situation is even more difficult if we are looking at Riesz transforms on bounded subsets (with Neumann-Dirichlet conditions), second order elliptic operators on Lipschitz domains and Riesz transforms on Riemannian manifolds, etc. For all such operators,  there is only a range of exponents $(p_0,q_0)$ where we have $L^p$ estimates for the semigroup $(e^{-tL})_t$ and its gradient with $p\in(p_0,q_0)$. Weighted estimates for such operators are more delicate, naturally restricted to these same ranges of $p$. We refer the reader to \cite{AM} for a recent ``survey'' about weighted estimates. 

\medskip

In this current work, we aim to combine these two fashionable problems and give a modern approach to singular non-kernel operators. This setting had been resisting to many of the ideas developed in recent years. Indeed, we are going to adapt the very recent approach of Lacey \cite{Lacey}, in order to be able to deal with non-integral singular operators. The main idea relies on defining a suitable maximal operator and then to control the operator by {\it sparse operators} (whose definition is modified with respect to the previous works). We describe our method in a very general setting given by a space of homogeneous type, equipped with a semigroup. However, we point out that even in the Euclidean case, our results are new since they do not rely on any pointwise regularity estimates of the kernel of the considered operators.
Moreover, we modify the maximal operator that we are going to use: instead of the maximal truncated operator used by Lacey in \cite{Lacey}, we use truncation in the `frequency' point of view (where the notion of `frequencies' has to be understood in terms of the semigroup). Simultaneously, we will use a slightly weaker notion of sparse operators, both of these facts will allow us to give a proof which is simpler than Lacey's proof. However, we are not able to recover the full $A_2$ result in its generality: indeed the assumptions we need, requires that the considered operator satisfies a suitable decomposition in the `frequency point of view' (see Remark \ref{rem:CZO}). As shown in Section \ref{sec:example}, that covers the main prototypes of operators. It is interesting to observe that the proof of these sharp weighted estimates can be substantially simplified in our situation and extended to operators whose kernel does not satisfy any regularity estimate. Recently in \cite{BCDH}, the authors have extended Lerner's approach for operator with kernels having $L^{p_0}$-$L^\infty$ regularity estimates (which corresponds to the $q_0=\infty$ as we will see in Subsection \ref{subsec:fourier}). We emphasise that this work is the first one where we are able to consider the case $q_0<\infty$ and where no regularity is required on the eventual `kernel', which (as shown in the examples in Section \ref{sec:example}) will allow us to deal with various situations in terms of operators and ambient spaces. 

\subsection{The setting}

Let $M$ be a locally compact separable metric space equipped with a Borel measure $\mu$, finite on compact sets and strictly positive on any non-empty open set.
For $\Omega$ a measurable subset of $M$, we shall often denote $\mu\left(\Omega\right)$ by $\left|\Omega\right|$.

For all $x \in M$ and all $r>0$, denote by $B(x,r)$ the open ball for the metric $d$ with centre $x$ and radius $r$.
%, and  by $V(x,r)$ its measure $|B(x,r)|$.  
For a ball $B$ of radius $r$ and  $\lambda>0$, denote by $\lambda B$   the ball concentric  with $B$ and with radius $\lambda r$.  We sometimes denote by $r(B)$ the radius of the ball $B$. Finally, we will use $u\lesssim v$ to say that there exists a constant $C$ (independent of the important parameters) such that $u\leq Cv$ and $u\simeq v$ to say that $u\lesssim v$ and $v\lesssim u$. Moreover, for $\Omega\subset M$ a subset of finite and non-vanishing measure and $f\in L^1_{\loc}(M,\mu)$, $\aver{\Omega} f \, d\mu=\frac{1}{|\Omega|} \int f \, d\mu$ denotes the average of $f$ on $\Omega$.  
We denote by $\calM$ the uncentered Hardy-Littlewood maximal operator. For $p \in [1,\infty)$, we abbreviate by $\calM_p$ the operator defined by $\calM_p(f):= [\calM(|f|^p)]^{1/p}$, $f \in L^p_{\loc}(M,\mu)$.

We shall assume  that   $(M,d,\mu)$ satisfies the volume doubling property, that is
  \begin{equation}\label{d}\tag{$V\!D$}
     |B(x,2r)| \lesssim  |B(x,r)|,\quad \forall~x \in M,~r > 0.
    \end{equation}
It follows that there exists  $\nu>0$  such that
     \begin{equation*}\label{dnu}\tag{$V\!D_\nu$}
      |B(x,r)|\lesssim \left(\frac{r}{s}\right)^{\nu} |B(x,s)|,\quad \forall~ x \in M,~r \ge s>0,
    \end{equation*}
which implies
     \begin{equation*}
      |B(x,r)|\lesssim \left(\frac{d(x,y)+r}{s}\right)^{\nu} |B(y,s)|,\quad \forall~ x,y \in M, ~r \ge s>0.
    \end{equation*}
An easy consequence of \eqref{d} is that  balls with a non-empty intersection and comparable radii have comparable measures.\\

We suppose that there exists an unbounded operator $L$ on $L^2(M,\mu)$ satisfying the following assumptions.

\begin{assum-L}
Let $L$ be an injective, $\theta$-accretive operator with dense domain $\mathcal D_2(L)$ in $L^2(M,\mu)$, where $0 \leq \theta < \pi/2$. Suppose that there exist two exponents $1\leq p_0 <2 < q_0\leq \infty$ such that for all balls $B_1,B_2$ of radius $\sqrt{t}$ 
\begin{align} \label{eq:OD-semigroup}
  \| e^{-tL} \|_{L^{p_0}(B_1) \to L^{q_0}(B_2)} 
 \lesssim  |B_1|^{-1/p_0} |B_2|^{1/q_0} e^{-c\frac{d(B_1,B_2)^2}{t}}.
\end{align}
\end{assum-L}

As a consequence, $L$ is a maximal accretive operator on $L^2(M,\mu)$, and therefore has a bounded $H^\infty$ functional calculus on $L^2(M,\mu)$. 
The assumption $\theta<\frac{\pi}{2}$ implies that $-L$ is the generator of an analytic semigroup $(e^{-tL})_{t>0}$ in $L^2(M,\mu)$ (see \cite{ADM,Kato} for definitions and further considerations).
The last part in the assumption means that the considered semigroup satisfies $L^{p_0}$-$L^{q_0}$ off-diagonal estimates in extension to $L^2$-$L^2$ Davies-Gaffney estimates. In situations where pointwise heat kernel bounds fail (see below for examples), this has turned out to be an appropriate replacement. \\

In this work, we study weighted estimates for non-integral singular operators satisfying some cancellation with respect to this operator. We consider a linear (or sublinear) operator $T$ satisfying the following properties:

\begin{assum} \label{assum-T}
\begin{itemize}
\item[$(a)$] $T$ is well-defined as a bounded operator in $L^2$;
\item[$(b)$] 
($L^{p_0}$-$L^{q_0}$ off-diagonal estimates) 
There exists 
$N_0 \in \N$ such that for all integers $N\geq N_0$ and all balls $B_1,B_2$ of radius $\sqrt{t}$ 
\begin{align*}
%  \| e^{-tL} \|_{L^{p_0}(B_1) \to L^{q_0}(B_2)} 
% &\lesssim  |B_1|^{-1/p_0} |B_2|^{1/q_0} e^{-c\frac{d(B_1,B_2)^2}{t}},\\
   \| T (tL)^Ne^{-tL}\|_{L^{p_0}(B_1) \to L^{q_0}(B_2)} \lesssim  |B_1|^{-1/p_0} |B_2|^{1/q_0} \left(1+\frac{d(B_1,B_2)^2}{t}\right)^{-\frac{\nu+1}{2}}.
 \end{align*}
\item[$(c)$]
There exists an exponent $p_1 \in [p_0,2)$ such that for all $x \in M$ and $r>0$
\begin{align*}
	\left(\aver{B(x,r)} |Te^{-r^2 L} f|^{q_0}\,d\mu \right)^{1/q_0}
	\lesssim \inf_{y \in B(x,r)} \calM_{p_1}(Tf)(y) + \inf_{y \in B(x,r)} \calM_{p_1}(f)(y).
\end{align*}
\end{itemize}
\end{assum}

Item $(b)$ encodes the fact that the operator $T$ has some cancellation property which interacts well with the cancellation of the considered semigroup. 
Item $(c)$ is a property which allows us to get off-diagonal estimates for the low-frequency part of the operator $T$. 
We point out that Items $(b)$ and $(c)$ are the main assumptions and were already used in numerous works to replace the notion of Calder\'on-Zygmund operators (see e.g. \cite{memoirs,ACDH} and references therein).

We will assume the above throughout the paper. We abbreviate the setting with $(M,\mu,L,T)$.

\subsection{Results}

Consider the setting $(M,\mu,L,T)$ satisfying the previous assumptions. Then we claim that such an operator satisfies weighted boundedness. Indeed, such operator satisfies the three following properties:
\begin{itemize}
\item $T$ is bounded on $L^2$;
\item for every $r>0$ and some integer $N$ large enough, $T(I-e^{-r^2L})^N$ satisfies $L^{p_0}$-$L^{q_0}$ off-diagonal estimates (outside the diagonal), see Corollary \ref{cor} for a precise statement;
\item $T$ satisfies the Cotlar type inequality Assumption \ref{assum-T} (c) for some $p_1<2$.
\end{itemize}
We then know from \cite[Theorems 1.1 and 1.2]{memoirs} (see also the earlier results in \cite{AM,BK3,ACDH})  that $T$ is bounded in $L^p$ for every $p\in(p_0,q_0)$. By \cite[Theorem 3.13]{AM}, $T$ also satisfies some weighted estimates: for every $p\in(p_0,q_0)$ and every weight $\omega \in A_{\frac{p}{p_0}} \cap RH_{\left(\frac{q_0}{p}\right)'}$ (see Section \ref{sec:weights} for a precise definition of this class of weights), the operator $T$ is bounded in $L^p_\omega$.  However, it is not clear from these previous results, how the quantity $ \|T\|_{L^p_\omega \to L^p_\omega}$ depends on the weight $\omega$. The methods used do not tend to give optimal estimates.

Our main result is the following:

\begin{theorem} \label{thm} Consider the setting $(M,\mu,L,T)$ as above.
For $p\in (p_0,q_0)$, there exists a constant $c_p$ such that for every weight $\omega \in A_{\frac{p}{p_0}} \cap RH_{\left(\frac{q_0}{p}\right)'}$
$$ \| T \|_{L^p_\omega \to L^p_\omega} \leq c_p  \left([\omega]_{A_{\frac{p}{p_0}}} [\omega]_{ RH_{\left(\frac{q_0}{p}\right)'}}\right)^\alpha,$$
with 
\begin{equation} \label{def:alpha}
 \alpha:=\max\left\{\frac{1}{p-p_0}, \frac{q_0-1}{q_0-p}\right\}.
 \end{equation}
In particular, by defining the specific exponent
$$ \mathfrak{p}:=1+\frac{p_0}{q_0'} \in (p_0,q_0),$$
we have $\alpha=\frac{1}{p-p_0}$ if $p \in(p_0, \mathfrak{p}]$, and $\alpha=\frac{q_0-1}{q_0-p}$ if $p \in [\mathfrak{p},q_0)$.
\end{theorem}

\begin{rem}
In the case $q_0=p_0'$, we have $\mathfrak{p}=2$ and obtain a sharp $A_2$ inequality with an exponent
$$\alpha=\frac{p_0}{2-p_0}.$$
\end{rem}

\begin{rem}
If $p_0=1$ and $q_0=\infty$, we obtain $\alpha=\max(1,1/(p-1))$ and so we reprove the $A_2$ conjecture for such operators. Note that we are then able to prove these sharp estimate in the case of the Riesz transform $T=\nabla L^{-1/2}$in the situation where this operator does not fit the Calder\'on-Zygmund framework (there is no pointwise regularity estimate of the full kernel), see Section \ref{sec:example}.
\end{rem}

\begin{rem} We also prove the optimality of such estimates (in terms of the growth with respect to the characteristic of the weight) for sparse operators, which are shown to control our operators. The optimality also still holds for the operator itself, if we know some `lower off-diagonal' estimates.
\end{rem}

\begin{rem} \label{rem:CZO} On the Euclidean space ${\mathbb R}^\nu$, consider the canonical heat semigroup and $T$ an `arbitrary' Calder\'on-Zygmund operator: a linear $L^2$-bounded operator with a kernel $K$ satisfying some regularity estimates
$$ \left|K(x,y)-K(z,y)\right|+\left|K(y,x)-K(z,x)\right| \lesssim \left(\frac{d(y,z)}{d(y,z)+d(x,y)}\right)^\varepsilon d(x,y)^{-\nu},$$
for some $\varepsilon>0$ and every points $x,y,z$ with $2d(y,z)\leq d(x,z)$.
Then it is well-known that $T$ is $L^{p}$-bounded for every $p\in(1,\infty)$. Consequently, we can check that Assumption \ref{assum-T} is satisfied for $p_0=1$ and any $q_0< \infty$ as large as we want. So unfortunately, it is unclear to us if our approach could recover the optimal $A_p$ estimates for arbitrary Calder\'on-Zygmund operators (which would correspond to $q_0=\infty$). It appears that Assumption \ref{assum-T} (b) describes an extra property on the operator $T$, a kind of suitable frequency decomposition / representation (as Fourier multipliers or paraproducts, ...). It is interesting to observe that under this extra property, we are going to detail an `elementary' proof of the sharp weighted estimates (simpler than all the existing proofs \cite{Lerner3,Lacey}, ...) which has also the very important property to support the extension to non-integral operators with no regularity property on the kernel.

We remark that this extra property already appeared in \cite[Theorem 3]{DM} where boundedness of the maximal operator $T^{\#}$ (see Section \ref{sec:maxop} for the definition) in the case  $q_0=\infty$ was shown, and that this is also the only place where we are using it. See also \cite[Remark p. 251]{DM}.  Moreover, as illustrated in Section \ref{sec:example}, this extra property is satisfied for the main prototype of Calder\'on-Zygmund operators.
\end{rem}

\section{Notations and Preliminaries on approximation operators}

\subsection{Notations}

For $p\in[1,\infty)$, a subset $E\subset M$ and $\lambda$ a measure on $E$, we write $L^p(E,d\lambda)$ for the Lebesgue space, equipped with the norm
$$ \|f\|_{L^p(E,d\lambda)} = \left(\int_E |f|^p d\lambda \right)^{1/p}.$$

For convenience, we forget $E$ if $E=M$ is the whole space and $\lambda$ if $\lambda=\mu$ is the underlying measure. So $L^p$ stands for $L^p(M,\mu)$.
For a positive function $\omega$, $L^p_\omega$ denotes the weighted Lebesgue space, equipped with the norm
$$ \|f\|_{L^p_\omega} = \left(\int_M |f|^p \omega d\mu \right)^{1/p}.$$

For a positive function $\rho:M \rightarrow (0,\infty)$, we identify the function $\rho$ with the measure $\rho d\mu$ in the sense that for every measurable subset $E\subset M$, we use
$$ \rho(E) = \int_{E} \rho d\mu.$$

For a ball $B$, we denote $S_0(B)=2B$ and $S_j(B)=2^{j+1}B \setminus 2^j B$, $j \geq 1$.
By extending the average notion to coronas, we denote 
$$
	\aver{S_j(B)} f\,d\mu = |2^jB|^{-1} \int_{S_j(B)} f \,d\mu.
$$

\subsection{Operator estimates}

The building blocks of our analysis will be the following operators derived from the semigroup $(e^{-tL})_{t>0}$. They serve as a replacement for Littlewood-Paley operators. 

Two different classes of elementary operators will be needed: one $(P_t)_{t>0}$ corresponding to an approximation of the identity at scale $\sqrt{t}$ commuting with the heat semigroup and $(Q_t)_{t>0}$ which satisfies some extra cancellation with respect to $L$.

\begin{definition} \label{def:Qt-Pt}
Let $N > 0$, and set $c_N=\int_0^{+\infty} s^{N} e^{-s} \,\frac{ds}{s}$.
 For $t>0$, 
define
\begin{equation} \label{def:Qt}
	Q_t^{(N)}:=c_N^{-1}(tL)^{N} e^{-tL}
\end{equation}
and 
\begin{equation} \label{def:Pt}
	P_t^{(N)}:=  \int_{1}^\infty Q^{(N)}_{st} \, \frac{ds}{s} = \phi_N(tL),
\end{equation}
with $\phi_N(x):= c_N^{-1}\int_x^{+\infty} s^{N} e^{-s} \,\frac{ds}{s}$,  $x\ge 0$.
\end{definition}

\begin{rem} \label{rem:Pt}
Let $p \in [p_0,q_0]$ with $p<\infty$ and $N>0$.
\begin{enumerate}
\item[(i)] 
Note that $P_t^{(1)}=e^{-tL}$ and $Q_t^{(1)}=tLe^{-tL}$. The two families of operators $(P_t^{(N)})_{t>0}$ and $(Q_t^{(N)})_{t>0}$ are related by $$ t\partial_t P_t^{(N)} =  tL\phi'_N(tL)= - Q_t^{(N)}. $$
\item[(ii)]
If $N$ is an integer, then $Q_t^{(N)}=(-1)^{N}c_N^{-1} t^{N} \partial_t^{N} e^{-tL}$, and $P_t^{(N)}=p(tL)e^{-tL}$, $p$ being a polynomial of degree $N-1$ with 
$p(0)=1$.
\item[(iii)] By $L^p$ analyticity of the semigroup and \eqref{eq:OD-semigroup}, we know that for every integer $N> 0$ and every $t>0$, $P_{t}^{(N)}$ and $Q_{t}^{(N)}$ satisfy off-diagonal estimates at the scale $\sqrt{t}$. See the arguments in e.g. \cite[Proposition 3.1]{HLMMY}.
\item[(iv)]
The operators $P_t^{(N)}$ and $Q_t^{(N)}$ are bounded in $L^p$, uniformly in $t>0$. See \cite[Theorem 2.3]{AM2}, taking into account (iii).
\end{enumerate}
\end{rem}

\begin{proposition}[Calder\'on reproducing formula] \label{prop:calderon} Let $N>0$ and $p\in(p_0,q_0)$. 
For every $f\in L^p $,
\begin{align} \label{limzero}
\lim_{t\to {0^+}} P_t^{(N)}f  =  f \quad & \textrm{in $L^p$},\\ \label{liminfty}
\lim_{t\to {+\infty}} P_t^{(N)}f  =  0 \quad & \textrm{in $L^p$},
\end{align}
and
\begin{equation} \label{calde}
 f = \int_0^{+\infty} Q_t^{(N)}f \, \frac{dt}{t} \quad \text{in}\ L^p.
\end{equation}
In particular, it follows that, as $L^p$-bounded operators we have the decomposition
\begin{equation}\label{along}
P_t^{(N)} = \textrm{Id}+ \int_0^t Q_s^{(N)} \, \frac{ds}{s}.
\end{equation}
\end{proposition}

\section{Examples and applications}
\label{sec:example}

Our assumptions on $L$ hold for a large variety of second order operators, for example uniformly elliptic operators in divergence form and Schr\"odinger operators with singular potentials on $\R^n$, or the Laplace-Beltrami operator on a Riemannian manifold. For more precise examples of $L$ and references, see Subsection \ref{subsec:Riesz} below. We give some examples of singular integral operators $T$ that fit into our setting. See also \cite{AM3}.

\subsection{Holomorphic functional calculus of $L$}
Let $0 \leq \theta< \sigma<\pi$, where $\theta$ denotes the angle of accretivity of $L$. 
Define the open sector in the complex plane of angle $\sigma$ by
\begin{align*}
%	S_\sigma & :=\{ z \in \C \,:\, |\arg z|\leq \sigma \} \cup \{0\};\\
	S_{\sigma}^o := \{z \in \C \,:\, z \neq 0, \ |\arg z| <\sigma \}.
\end{align*}
Denote by $H(S_\sigma^o)$ the space of all holomorphic functions on $S_\sigma^o$, and let
$$
	H^\infty(S_\sigma^o):=\{ \varphi \in H(S_\sigma^o):\,\|\varphi\|_\infty <\infty\}.
$$
By our assumptions, $L$ has a bounded $H^\infty$ functional calculus on $L^2$. It was shown in \cite{BK3} that under the assumption \eqref{eq:OD-semigroup}, the functional calculus can be extended to $L^p$ for $p \in (p_0,q_0)$.

We now obtain the following weighted version. Let $\sigma>\theta$, and let $\varphi \in H^\infty(S_\sigma^o)$. Set $T=\varphi(L)$. We check  Assumption \ref{assum-T}. Item (a) is a restatement of the fact that $L$ has a bounded $H^\infty$ functional calculus on $L^2$. Since $T$ commutes with $e^{-r^2L}$, Item (c) can be obtained as a consequence of \eqref{eq:OD-semigroup} (we do not detail this here, similar estimates are done in the sequel). Finally for large enough $N$, by adapting \cite[Lemma 3.6]{AMcR} one can show that $\varphi(L)(tL)^Ne^{-tL}$ satisfies $L^{q_0}$-$L^{q_0}$ off-diagonal estimates. Combining this with $L^{p_0}$-$L^{q_0}$ off-diagonal estimates for $e^{-tL}$ gives (b).
We therefore have

\begin{theorem}
Let $p \in (p_0,q_0)$, and $\omega \in A_{\frac{p}{p_0}} \cap RH_{\left(\frac{q_0}{p}\right)'}$.
The operator $L$ has a bounded holomorphic functional calculus in $L^p_\omega$ with, for every $\sigma>\theta$, 
$$ \| \varphi(L) \|_{L^p_\omega \to L^p_\omega} \leq c_{p,\sigma}  \left([\omega]_{A_{\frac{p}{p_0}}} [\omega]_{ RH_{\left(\frac{q_0}{p}\right)'}}\right)^\alpha \|f\|_\infty,$$
for all $\varphi \in H^\infty(S_\sigma^o)$, and 
$\alpha$ as defined in \eqref{def:alpha}. 
\end{theorem}

\subsection{Riesz transforms}
\label{subsec:Riesz}

The $L^p$ boundedness of Riesz transforms on manifolds has been widely studied in recent years. 
We refer the reader to \cite{BF} for a recent work and references therein for more details about such operators.

Several situations fit into our setting, we can consider specific operators, or specific ambient spaces or both. Let us give some examples, more can be studied, like Riesz transforms on bounded domains, or associated with Schr\"odinger operators.

\begin{itemize}

\item {\bf Dirichlet forms.} 

Let $(M,d,\mu)$ be a complete space of homogeneous type as above. Consider a self-adjoint operator $L$ on $L^2$ and consider $\mathcal{E}$ the quadratic form associated with $L$, that is
$${\mathcal E}(f,g)=\int_M f Lg\,d\mu.$$ 
If ${\mathcal E}$ is a strongly local and regular Dirichlet form (see \cite{FOT, GSC} for precise definitions) with a carr\'e du champ structure, then with $\Gamma$ being equal to this carr\'e du champ operator, assume that the Poincar\'e inequality $(P_2)$ holds, that is,
\begin{equation*} %\tag{$P_2$} 
 \left( \aver{B} | f - \aver{B} f d\mu |^2 d\mu \right)^{1/2} \lesssim r \left(\aver{B} d\Gamma(f,f)\right)^{1/2},  \label{P2}
\end{equation*}
for every  $f\in {\mathcal D}(\mathcal{E})$ and every ball $B\subset M$ with radius $r$. 

If the heat semigroup $(e^{-tL})_{t>0}$ and its carr\'e du champ $(\sqrt{t} \Gamma e^{-tL})_{t>0}$ satisfy $L^{p_0}$-$L^{q_0}$ off-diagonal estimates, then it can be checked that Assumption \ref{assum-T} is satisfied for the Riesz transform (see \cite{ACDH}) \footnote{It is known that the assumed Poincar\'e inequality $(P_2)$ self-improves into a Poincar\'e inequality $(P_{p_1})$ for some $p_1<2$ (see \cite{KZ}), which allows us to check Item $(c)$ of Assumption \ref{assum-T}.}
$$ {\mathcal R}:= \Gamma L^{-1/2} = c_k \Gamma\left[\int_0^\infty  (tL)^k e^{-tL} \frac{dt}{\sqrt{t}}\right],$$
for some numerical constant $c_k$ and every integer $k\geq 1$.

In particular for $p_0=1$ and $q_0=\infty$, we get the following result.

\begin{theorem} Consider ${\mathcal R}$ the Riesz transform in one of the following situations:
\begin{itemize}
\item Euclidean space or any doubling Riemannian manifold with bounded geometry and nonnegative Ricci curvature (see \cite{LY});
\item In a convex doubling subset of ${\mathbb R}^\nu$ with the Laplace operator associated with Neumann boundary conditions (see \cite{YW}).
\end{itemize}
Then for every $p\in (1,\infty)$ and every weight $\omega \in A_p$ we have
$$ \|{\mathcal R}\|_{L^p_\omega \to L^p_\omega} \lesssim [\omega]_{A_p}^\alpha \qquad \textrm{with} \qquad \alpha=\max\{1,\frac{1}{p-1}\}.$$
\end{theorem}

Note that in these situations, we only have Lipschitz regularity of the heat kernel but the full kernel of the Riesz transform does not satisfy any pointwise regularity estimate and so does not fit into the class of Calder\'on-Zygmund operators (as previously studied in \cite{Lerner3, Lacey}).

\item {\bf Second order divergence form operators.}

Consider $(M,d,\mu)$ a doubling Riemannian manifold, equipped with the Riemannian gradient $\nabla$ and its divergence operator $\dive = \nabla^*$. To $A = A(x)$ a complex, bounded, measurable matrix-valued function, defined on $M$ and satisfying the ellipticity (or accretivity) condition $\text{Re}(A(x)) \geq \kappa I >0$ a.e., we may define a second order divergence form operator
$$ L=L_A f :=- \dive (A\nabla f).$$
Then $L$ is sectorial and satisfies the conservation property but may not be self-adjoint. 

Assume that the Poincar\'e inequality $(P_2)$ holds on $(M,d,\mu)$. If the semigroup $(e^{-tL})_{t>0}$ and its gradient $(\sqrt{t} \nabla e^{-tL})_{t>0}$ satisfy $L^{p_0}$-$L^{q_0}$ off-diagonal estimates, then it can be checked that Assumption \ref{assum-T} is satisfied for the Riesz transform
$$ {\mathcal R}:= \nabla L^{-1/2} = c_k \int_0^\infty \nabla (tL)^k e^{-tL} \frac{dt}{\sqrt{t}}.$$

We refer the reader to \cite{memoirs} for a precise study in the Euclidean setting of the exponents $p_0,q_0$ depending on the matrix-valued map $A$. For example, we have $p_0=1$ and $q_0=\infty$ if $\nu=1$.
\end{itemize}

\subsection{Paraproducts associated with $L$}

In all this paragraph we assume that the semigroup satisfies the conservation property, which means that $e^{-tL}1=1$ for every $t>0$, as well as the fact that the semigroup is supposed to have a heat kernel with pointwise Gaussian bounds (which corresponds to $L^1$-$L^\infty$ estimates).

\subsubsection*{Paraproducts with a $BMO$-function}

In recent works \cite{BT1,F}, several paraproducts have been studied in the context of a semigroup. They allow us to have (as well-known in the Euclidean space) a decomposition of the pointwise product with two paraproducts and a resonent term (we also refer the reader to \cite{BBF} for some applications of such paraproducts in the context of paracontrolled calculus for solving singular PDEs). Moreover, $BMO$ spaces adapted to such a framework have also been the aim of numerous work, so it is natural (as done in the Euclidean setting) to study the linear operator given by the paraproduct of a $BMO$-function.

Let us recall some definitions. A $BMO_L$-function is a locally integrable function $f\in L^1_{loc}$ such that
$$ \|f\|_{BMO_L} := \sup_{B} \left(\aver{B} \left| f - e^{-r^2L} f \right|^2 \, d\mu \right)^{1/2},$$
where we take the supremum over all balls $B$ with radius $r>0$. 
Such $BMO$-spaces satisfy `standard' properties, as for example John-Nirenberg inequality, $T(1)$-theorem, ... 
In particular it is known (see \cite{BZ,BM}) that since the semigroup satisfies $L^{1}$-$L^{\infty}$ off-diagonal estimates, the norm in $BMO_L$ can be built through any $L^p$-oscillation for any $p\in(1,\infty)$ and the corresponding norms are equivalent. For some integer $k$, the paraproduct under consideration is
$$ \Pi_g(f) = \int_0^\infty Q_t^{(k)}\left[Q_t^{(k)}g \cdot P_t^{(k)}f\right] \, \frac{dt}{t}.$$
Using square function estimates, then we know that for $g\in BMO_L$, $Q_t^{(k)}g$ is uniformly bounded in $L^\infty$ and so that $\Pi_g$ is $L^2$-bounded. Assumption \ref{assum-T} $(b)$ and $(c)$ are also satisfied with $p_0=1$ and $q_0=\infty$ (see details in the above references) and so we may apply Theorem \ref{thm} to the previous paraproduct, for $g\in BMO_L$.

\subsubsection*{Algebra property for fractional Sobolev spaces}

In a recent work \cite{BCF2}, some paraproducts associated with such a framework involving a heat semigroup have been very useful in order to study the algebra property for fractional Sobolev spaces. We refer to \cite{BCF2} for more details and more references for other paraproducts associated with a semigroup. 
Then up to some constant $c_N$, we have the product decomposition for two functions
$$ fg = \Pi_g(f) + \Pi_g(f)$$
with the paraproduct defined by
$$ \Pi_g(f) = \int_0^\infty Q_t^{(N)} f \cdot P_t^{(N)} g \, \frac{dt}{t}.$$

Fix a function $g\in L^\infty$, then for $\alpha \in(0,1)$ we are looking for $\dot L^p_{\alpha}$-boundedness of $\Pi_g$, which corresponds to  $L^p$ boundedness of $T:=L^{\alpha/2} \Pi_g L^{-\alpha/2}$. In \cite{BCF2}, we gave different situations / criterions under which Assumption \ref{assum-T} is satisfied. Mainly we considered the condition, introduced in \cite{ACDH}, for some $p\in(2,\infty]$
\begin{equation} \label{Gp}
\sup_{t>0} \|\sqrt{t}|\Gamma e^{-tL} |\|_{p\to p} <+\infty \tag{$G_p$},
\end{equation}
where $\Gamma$ is the carr\'e du champ associated with the operator $L$. In this way, we may apply Theorem \ref{thm} to $T$ and obtain sharp algebra property for weighted fractional spaces, sharp with respect to the weight. We obtain the following estimates.

\begin{theorem}  Let $(M,d,\mu, {\mathcal E})$  be a doubling metric  measure   Dirichlet space with a `carr\'e du champ' (see \cite{BCF2} for more details) and assume that the heat semigroup $e^{-tL}$ has a heat kernel with usual pointwise Gaussian estimates. For some $s\in(0,1)$ and $p\in(1,\infty)$, consider the following weighted Leibniz rule: for every weight $\omega$ and all functions $f,g\in \{h\in L^\infty,\ L^{s/2}(h) \in L^p_\omega\}$
\begin{equation} \| L^{s/2} (fg) \|_{L^p_\omega} \lesssim c(\omega) \left(\|L^{s/2} f\|_{L^p_\omega} \|g\|_{\infty} +  \|f\|_{\infty} \|L^{s/2}g\|_{L^p_\omega} \right).\label{A} 
\end{equation}
\begin{itemize}
 \item[(a)] Then \eqref{A} is valid for $p\in(1,2)$ and $s\in(0,1)$ with every weight $\omega \in 
A_{p} \cap RH_{ \left(\frac{2}{p}\right)'}$ and a constant
 $$ c(\omega) = \left([\omega]_{A_p} [\omega]_{ RH_{\left(\frac{2}{p}\right)'}}\right)^\alpha \qquad \textrm{with} \qquad \alpha:=\max\left\{\frac{1}{p-1}, \frac{1}{2-p}\right\}.$$
\item[(b)] Under $(G_q)$ for some $q\in(2,\infty)$ then   \eqref{A} is valid for $p\in(1,q)$ and $s\in(0,1)$ with for $q^-\in(p,q)$ and every weight $\omega \in A_p \cap RH_{ \left(\frac{q^-}{p}\right)'}$
 and a constant
 $$ c(\omega) = \left([\omega]_{A_p} [\omega]_{ RH_{\left(\frac{q^-}{p}\right)'}}\right)^\alpha \qquad \textrm{with} \qquad \alpha:=\max\left\{\frac{1}{p-1}, \frac{q^--1}{q^--p}\right\}.$$
\item[(c)]
Under $(G_\infty)$ then   \eqref{A} is valid for $p\in(1,\infty)$ and $s\in(0,1)$ with every weight $\omega \in A_p$
 and a constant
 $$ c(\omega) = [\omega]_{A_p}^\alpha \qquad \textrm{with} \qquad \alpha:=\max\left\{\frac{1}{p-1}, 1\right\}.$$
\end{itemize}
\end{theorem}
 
Other estimates can be obtained by combining the result of this paper with the other estimates of \cite{BCF2}.

\subsection{Fourier multipliers} \label{subsec:fourier}

Let us also explain how we can recover the very recent results of \cite{BCDH}. The main linear result \cite[Theorem C]{BCDH} fits into our framework and corresponds to the particular case $q_0=\infty$. Let us focus on the application to linear Fourier multipliers.

Consider $m$ a linear symbol on ${\mathbb R}^\nu$ satisfying the H\"ormander condition $M(s,l)$,  which is
$$ \sup_{R>0} \left(R^{s|\alpha|-\nu} \int_{R\leq |\xi|\leq 2R} \left|\partial^\alpha_\xi m(\xi)\right|^s \, d\xi \right)^{1/s},$$
for all $|\alpha|\leq l$, some $s\in(1,2]$ and $l\in(\nu/s,\nu)$.
To this symbol, we associate the linear Fourier multiplier as 
$$ T(f)=T_m(f):= x \mapsto \int e^{ix.\xi} m(\xi) \widehat{f}(\xi)\, d\xi.$$ For every $r\in(\nu/l,\infty)$, \cite[Lemma 5.2]{BCDH} shows that the kernel of $T$ satisfies some $L^{r}$-$L^{\infty}$ regularity off-diagonal estimates. So consider a smooth function $\psi$ so that $\widehat{\psi}$ is supported on $B(0,4) \setminus B(0,1)$ and well-normalized so that $\int_0^\infty \widehat{\psi}(t\xi) \frac{dt}{t}=1$ for every $\xi$. Then with the elementary operators 
$$ T_t(f) = x \mapsto \int e^{ix.\xi} m(\xi) \widehat{\psi}(t\xi) \widehat{f}(\xi)\, d\xi,$$
it can be proved that Assumption \ref{assum-T} is satisfied for $p_0=r$ and $q_0=\infty$. Consequently Theorem \ref{thm} allows us to regain \cite[Theorem 5.3 (a)]{BCDH}. Moreover, since $T$ is self-adjoint, by duality we also deduce that the kernel of $T$ satisfies some $L^1$-$L^{r'}$ off-diagonal estimates. Similarly, one can then show that Assumption \ref{assum-T} is satisfied for $p_0=1$ and $q_0=r'$. Consequently Theorem \ref{thm} allows us to regain \cite[Theorem 5.3 (b)]{BCDH}.

So we regain the same full result as in \cite[Theorem 5.3]{BCDH}, with the exact same behavior of the weighted estimates with respect to the weight.

The same comparison can be done for the linear part of their main result \cite[Theorem C]{BCDH}. Under their assumptions $(H1)$ and $(H2)$, our Assumption \ref{assum-T} is satisfied with $q_0=\infty$. We leave the details to the reader.

\section{Unweighted boundedness of a certain maximal operator}
\label{sec:maxop}

Before introducing and studying a certain maximal operator related to $T$, we first explain some technical details for off-diagonal estimates.

\subsection{Off-diagonal estimates} 

We fix an integer $N>N_0$ (with $N_0$ as in Assumption \ref{assum-T}) and write for $t>0$
$$ T_t:= T Q_t^{(N)}.$$ 
Let $p\in(p_0,q_0)$. The Calder\'on reproducing formula (see Proposition \ref{prop:calderon}) gives the identity
$$ \textrm{Id} = \int_0^\infty Q_t^{(N)} \, \frac{dt}{t}$$ in $L^p$.
Since $T$ is assumed to be sublinear, we can decompose the operator for $f\in L^p$ into
\begin{align}
 |T(f)| \leq \int_0^\infty |T_t(f)| \, \frac{dt}{t}. \label{eq:decomposition}
\end{align}

Fix $t>0$ and the elementary operator $T_t$. From  Assumption \ref{assum-T} (b) we know that $T_t$ satisfies $L^{p_0}$-$L^{q_0}$ off-diagonal estimates at the scale $\sqrt{t}$. Then consider a ball $B$ of radius $r>0$ with $r\leq \sqrt{t}$ and its dilated ball $\tilde{B}:=\frac{\sqrt{t}}{r}B$. We have $B\subset \tilde B$ and $|\tilde B| \lesssim \left(\frac{\sqrt{t}}{r}\right)^\nu |B|$ and so
\begin{equation} \label{OD-I}
	\left(\aver{B} |T_t f|^{q_0}\,d\mu\right)^{1/q_0} 
	\lesssim \left(\frac{\sqrt{t}}{r}\right)^{\nu/q_0} \sum_{j \geq 0} 2^{-j(\nu+1)} \left(\aver{S_j(\tilde{B})} |f|^{p_0}\,d\mu\right)^{1/p_0}.
\end{equation}

\begin{lemma} \label{lemma} Consider three parameters $r,\eps,t>0$. Let $N \in \N$ with $N>\max(3\nu/2+1,N_0)$. 
%the following estimates for $T_t(I-e^{-r^2L})^N$:
\begin{itemize}
\item If $r^2<\eps <t$, we have for every ball $B_r$ of radius $r$ and $B_{\sqrt{\eps}}=\frac{\sqrt{\eps}}{r} B_{r}$ the dilated ball,
$$ \left( \aver{B_{\sqrt{\eps}}} \left| T_t(I-e^{-r^2L})^N f \right|^{q_0} \, d\mu \right)^{1/q_0} \lesssim \left(\frac{r^2}{t}\right)^{N/2} \sum_{\ell \geq 0} 2^{-\ell(\nu+1)}  \left( \aver{2^\ell B_{r}} \left|f \right|^{p_0} \, d\mu \right)^{1/p_0};$$
\item If $t<\eps < r^2$, we have for every ball $B_r$ of radius $r$, $j\geq 3$, every ball $B_{\sqrt{\eps}}$ of radius $\sqrt{\epsilon}$ included in $S_j(B_r)$, and every function $f$ supported on $B_r$,
$$ \left(\aver{B_{\sqrt{\eps}}} \left| T_t(I-e^{-r^2L})^N f \right|^{q_0} \, d\mu \right)^{1/q_0} \lesssim 2^{-j(\nu+1)} \left(\frac{t}{r^2}\right)^{1/2}  \left( \aver{B_r} \left|f \right|^{p_0} \, d\mu \right)^{1/p_0}.$$ 
\end{itemize}
The same estimates are true for $T_t(I-P_{r^2}^{(N)})$ in place of  $T_t(I-e^{-r^2L})^N$.
\end{lemma}

\begin{proof} Consider the first case $r^2<\eps <t$. 
We show the result for $T_t(I-P_{r^2}^{(N)})$, and then explain how to modify the proof in the case of $T_t(I-e^{-r^2L})^N$. 
By definition of $P_{r^2}^{(N)}$, we have
$$
	I-P_{r^2}^{(N)}=\int_0^{r^2} Q_s^{(N)} \,\frac{ds}{s}.
$$
Hence,
$$ \left( \aver{B_{\sqrt{\eps}}} \left| T_t(I-e^{-r^2L})^N f \right|^{q_0} \, d\mu \right)^{1/q_0} \lesssim \int_0^{r^2} \left( \aver{B_{\sqrt{\eps}}} \left| T_t Q_s^{(N)} f \right|^{q_0} \, d\mu \right)^{1/q_0} \, \frac{ds}{s}.$$
Note that $T_t=T Q_t^{(N)}$ and $Q_t^{(N)} Q_s^{(N)} = \left(\frac{s}{s+t}\right)^N Q_{s+t}^{(2N)}$ (up to a numerical constant) as well as $s+t \simeq t$.
Using Assumption \ref{assum-T} (b) and \eqref{OD-I} for $s<r^2$ with $r^2<\eps<t$, we obtain  that
\begin{align*} 
\left( \aver{B_{\sqrt{\eps}}} \left| T_t Q_s^{(N)} f \right|^{q_0} \, d\mu \right)^{1/q_0}  
% & \lesssim \left(\frac{s}{t}\right)^{N} \left( \aver{B_{\sqrt{\eps}}} \left| T_t(tL)^Ne^{-sL} f \right|^{q_0} \, d\mu \right)^{1/q_0} \\
& \lesssim \left(\frac{s}{t}\right)^{N}  \left( \aver{B_{\sqrt{\eps}}} \left| T_{s+t} f \right|^{q_0} \, d\mu \right)^{1/q_0} \\
& \lesssim \left(\frac{s}{t}\right)^{N} \left(\frac{t}{\eps}\right)^{\nu/2q_0} \sum_{\ell \geq 0} 2^{-\ell (\nu+1)} \left( \aver{2^\ell B_{\sqrt{t}}} \left|  f \right|^{p_0} \, d\mu \right)^{1/p_0},
\end{align*}
where we used that $s+t \lesssim t$. 
Since $s<r^2<\eps$, we can estimate $\left(\frac{s}{t}\right)^{N} \left(\frac{t}{\eps}\right)^{\nu/2q_0}$ by $\left(\frac{s}{t}\right)^{N-\nu/2q_0}$, and then deduce that with $k\geq 0$ such that $2^k r \simeq \sqrt{t}$
\begin{align*}
\left(\frac{s}{t}\right)^{N-\frac{\nu}{2q_0}} 2^{-\ell (\nu+1)} \left( \aver{2^\ell B_{\sqrt{t}}} \left|  f \right|^{p_0} \, d\mu \right)^{1/p_0}
& \lesssim \left(\frac{r^2}{t}\right)^{N/2} \left(\frac{s}{t}\right)^{\frac{N}{2}-\frac{\nu}{2q_0}}  2^{-\ell (\nu+1)}\left( \aver{2^\ell B_{\sqrt{t}}} \left|  f \right|^{p_0} \, d\mu \right)^{1/p_0} \\
& \lesssim \left(\frac{r^2}{t}\right)^{N/2} \left(\frac{s}{r^2}\right)^{\frac{\nu+1}{2}} 2^{-(\ell+k)(\nu+1)}  \left( \aver{2^{\ell+k} B_{r}} \left|  f \right|^{p_0} \, d\mu \right)^{1/p_0},
\end{align*}
where we used that $N$ is sufficiently large such that $\frac{N}{2}-\frac{\nu}{2q_0} > \frac{\nu+1}{2}$.
We then conclude by summing over $\ell$ and integrating in $s \in (0,r^2)$.

In the second case when $t<\eps < r^2$,  we follow the same reasoning: with $\tau =\max(s,t)$
\begin{align*} 
\left( \aver{B_{\sqrt{\eps}}} \left| T_t Q_s^{(N)} f \right|^{q_0} \, d\mu \right)^{1/q_0}  & \lesssim 
 \left(\frac{\min(s,t)}{\tau}\right)^{N-\frac{\nu}{2q_0}} \left(\frac{\tau}{2^{2j} r^2}\right)^{(\nu+1)/2} \left(\frac{r^2}{\tau}\right)^{\frac{\nu}{2p_0}} \left( \aver{B_r} \left|  f \right|^{p_0} \, d\mu \right)^{1/p_0} \\
  & \lesssim 
 \left(\frac{\min(s,t)}{\max(s,t)}\right)^{\frac{N}{2}-\frac{\nu}{2q_0}} \left(\frac{t}{r^2}\right)^{1/2}2^{-j(\nu+1)} \left( \aver{B_r} \left|  f \right|^{p_0} \, d\mu \right)^{1/p_0},
\end{align*}
where we used that $N>\nu+1$.
We may now integrate over $s$ and obtain the desired result.

The modifications required for the case $T_t(I-e^{-r^2L})^N$ are straightforward. 
Define $\psi_s^{(N)}(L)=N(1-e^{-sL})^{N-1} (sL)e^{-sL}$. Then $\psi_s^{(N)}(L)=-s\partial_s (I-(I-e^{-sL})^N)$ and 
$$
	(I-e^{-r^2L})^N = \int_0^{r^2} \psi_s^{(N)}(L) \,\frac{ds}{s}.
$$
Now we can also write $Q_t^{(N)}\psi_s^{(N)}(L) = \left(\frac{\min(s,t)}{\max(s,t)}\right)^N \Theta_{s,t}$ with some operator $\Theta_{s,t}$ satisfying $L^{p_0}$-$L^{q_0}$ off-diagonal estimates, and conclude as above. 
\end{proof}

Considering the particular case $\epsilon=r^2$, we may integrate in $t$ the previous two inequalities and from \eqref{eq:decomposition}, deduce the following result

\begin{coro} \label{cor} For an integer $N>\max(3\nu/2+1,N_0)$ and $r>0$, $T(I-e^{-r^2L})^N$ satisfies $L^{p_0}$-$L^{q_0}$ (strictly) off-diagonal estimates at the scale $r>0$: for $B_1,B_2$ two balls of radius $r>0$ with $d(B_1,B_2)>4r$ then for every function $f$ supported on $B_1$ we have
$$ \left(\aver{B_{2}} \left| T(I-e^{-r^2L})^N f \right|^{q_0} \, d\mu \right)^{1/q_0} \lesssim \left(1+\frac{d(B_1,B_2)}{r}\right)^{-(\nu+1)}  \left( \aver{B_1} \left|f \right|^{p_0} \, d\mu \right)^{1/p_0}.$$ 
\end{coro}

\subsection{Maximal operator}

We now fix an integer $N>\max(3\nu/2+1,N_0)$ (and all the implicit constants may depend on it).

\begin{definition}
 Define the maximal operator $T^{\#}$ of $T$ by
$$
	T^{\#}f(x) = \sup_{\substack{B \ball \\ B \ni x}} 
	\left(\aver{B} \left| T \int_{r(B)^2}^{\infty}  Q^{(N)}_t f\,\frac{dt}{t} \right|^{q_0}\,d\mu\right)^{1/q_0},  \qquad x \in M,
$$
for $f \in L^{q_0}_{\loc}$. 
\end{definition} 
By definition of $P^{(N)}_t:=\int_{1}^\infty Q^{(N)}_{st} \, \frac{ds}{s}$, we then have
$$
	T^{\#}f(x) = \sup_{\substack{B \ball \\ B \ni x}} 
	\left(\aver{B} \left| T P^{(N)}_{r(B)^2}f \right|^{q_0}\,d\mu\right)^{1/q_0},  \qquad x \in M,
$$
for $f \in L^{q_0}_{\loc}$.

\begin{lemma} Consider $(u_\eps)_\eps$ a sequence of $L^2$-functions which converges (in $L^2$) to some function $u \in L^2$, when $\eps$ tends to $0$. Then for almost every $x\in M$ we have
$$ |u(x)| \leq \liminf_{\eps \to 0}  \  \aver{B(x,\eps)} |u_\eps| \, d\mu.$$
\end{lemma}

\begin{proof} 
Due to the Lebesgue differentiation lemma, we know that
$$ |u(x)| \leq \liminf_{\eps \to 0} \   \aver{B(x,\eps)} |u| \, d\mu.$$
Then we split, as follows
$$  \aver{B(x,\eps)} |u| \, d\mu \leq \aver{B(x,\eps)} |u_\eps| \, d\mu + \aver{B(x,\eps)} |u-u_\eps| \, d\mu.$$
The second part is pointwisely bounded by $\calM[u_\epsilon-u](x)$ which converges in $L^2$ to $0$ (due to the $L^2$-boundedness of the maximal function), which allows us to conclude the proof.
\end{proof}

As a consequence of the previous lemma with the $L^2$-boundedness  of $T$ and Proposition \ref{prop:calderon}, we deduce the following result:

\begin{coro} For every function $f \in L^2$  we have almost everywhere
$$ \left|T(f)\right| \leq T^{\#}(f).$$
\end{coro}

\begin{proposition} \label{prop:op-max} The sublinear operator $T^{\#}$ is of weak-type $(p_0,p_0)$ and bounded in $L^p$ for every $p\in (p_0,2]$.
\end{proposition}

\begin{rem} In the definition of the maximal operator, the previous boundedness still holds if we replace the average on the ball $B$ by any average on $\lambda B$ for some constant $\lambda>1$. In such case, the implicit constants will depend on $\lambda$.
\end{rem}

\begin{proof}
\noindent
\textbf{Step 1:} $L^2$-boundedness of $T^{\#}$. \\
We first claim that $T^{\#}$ satisfies the following Cotlar type inequality ($p_1 \in [p_0,2)$ is introduced in Assumption \ref{assum-T}): 
\begin{equation} \label{eq:Cotlar}
	T^{\#}f(x) \lesssim \calM_{p_1}(Tf)(x) + \calM_{p_1}f(x), \qquad x \in M. 
\end{equation}
Indeed, 
$$ T^{\#}f(x) = \sup_{\substack{B \ball \\ B \ni x}}  \left(\aver{B} \left| T P^{(N)}_{r(B)^2}f \right|^{q_0}\,d\mu\right)^{1/q_0},$$
and since $N$ is an integer, we have by Remark \ref{rem:Pt}, with $r=r(B)$, that
$$ P^{(N)}_{r^2} = p(r^2 L) e^{-r^2L},$$
with $p$ a polynomial function. 
We then split 
$$  T P^{(N)}_{r^2} = \left(T e^{-\frac{r^2}{2}L} \right) \left( p(r^2L) e^{-\frac{r^2}{2}L}\right).$$
By Assumption \ref{assum-T} $(c)$, $T e^{-\frac{r^2}{2}L}$ satisfies some $L^{p_1}$-$L^{q_0}$ estimates and by Assumption \ref{assum-T} $(b)$ and Lemma \ref{lemma}, both $T \big(I-P^{(N)}_{r^2}\big)$ and $p(r^2L) e^{-\frac{r^2}{2}L}$ satisfy $L^{p_1}$-$L^{p_1}$ off-diagonal estimates at the scale $r$. We may compose these two estimates in order to obtain similar estimates as Assumption \ref{assum-T} $(c)$ for $TP^{(N)}_{r^2}$ and then directly obtain \eqref{eq:Cotlar}.

This in particular implies that $T^{\#}$ is bounded on $L^2$, since $T$ is bounded on $L^2$ by assumption, and the Hardy-Littlewood maximal operator $\calM_{p_1}$ is bounded on $L^2$ as $p_1<2$. 

In a second step, we now use the extrapolation method of \cite{memoirs,BK3} to show that $T^{\#}$ is of weak type $(p_0,p_0)$, which by interpolation with the $L^2$-boundedness will conclude the proof of the proposition.

\textbf{Step 2:} Weak-type $(p_0,p_0)$ of $T^{\#}$.\\
 We apply \cite[Theorem 1.1]{memoirs} (see also \cite{BK3}).
 As shown in Step 1, $T^{\#}$ is bounded on $L^2$. 
By assumption, we know that $(e^{-tL})_{t>0}$ satisfies $L^{p_0}$-$L^2$ off-diagonal estimates. It remains to show that 
$
	T^{\#}(I-e^{-tL})^N
$
satisfies $L^{p_0}$-$L^2$ off-diagonal estimates (not including the diagonal), where we will use (for convenience but it could be chosen differently) the same integer $N$ as the one defining the maximal operator, which is chosen sufficiently large. 
More precisely, for a ball $B \subseteq M$ of radius $r$ and a function $b \in L^{p_0}$ with $\supp b \subseteq B$, we will show that 
\begin{equation} \label{eq:OD-est-extrapol}	
	|2^{j+1}B|^{-1/2} \|T^{\#}(I-e^{-r^2L})^{N} b\|_{L^2(S_j(B))}
	\lesssim c(j) |B|^{-1/p_0} \|b\|_{L^{p_0}(B)}, \qquad j \geq 3, 
\end{equation}
with coefficients $c(j)$ satisfying $\sum_{j \geq 2} c(j) 2^{\nu j} <\infty$. \\

For $x \in M$ and $\eps>0$, denote by $B_{x,\eps}$ a ball of radius $\sqrt{\eps}$ containing $x$. Then recall that $T_t=TQ_t^{(N)}$ and 
\begin{align*}
	T^{\#}\big[(I-e^{-r^2L})^{N} b\big](x) 
	\leq \sup_{\eps>0} \left(\aver{B_{x,\eps}} \left| T \int_\eps^\infty Q_t^{(N)} \big[(I-e^{-r^2L})^{N} b\big]\,\frac{dt}{t} \right|^{q_0}\,d\mu \right)^{1/q_0}.
\end{align*}
Let $x \in S_j(B)$ for $j \geq 3$ and consider first the case $r^2<\eps$.
Applying Lemma \ref{lemma} (Part 1), we then deduce that
$$ \left(\aver{B_{x,\eps}} \left| T_t (I-e^{-r^2 L})^{N} b \right|^{q_0}\,d\mu \right)^{1/q_0} \lesssim \left(\frac{r^2}{t}\right)^{N/2} \left(1+\frac{d(B,B_{x,\eps})^2}{t} \right)^{-\frac{\nu+1}{2}} \left( \aver{B} |b|^{p_0} \, d\mu \right)^{1/p_0}.$$
Since $\epsilon<t$, it follows that
either $2^j r \geq \sqrt{t}$ and then $d(B, B_{x,\eps}) \simeq 2^j r$ or $2^j r\leq \sqrt{t}$ and then $d(B,B_{x,\eps})\leq 2 \sqrt{t}$. So in both situations, we have
$$ \left(1+\frac{d(B,B_{x,\eps})^2}{t} \right) \simeq \left(1+\frac{4^j r^2}{t} \right).$$
Consequently, we get
$$ \left(\aver{B_{x,\eps}} \left| T_t  (I-e^{-r^2 L})^{N} b \right|^{q_0}\,d\mu \right)^{1/q_0} \lesssim \left(\frac{r^2}{t}\right)^{N/2} \left(1+\frac{4^j r^2}{t} \right)^{-\frac{\nu+1}{2}} \left( \aver{B} |b|^{p_0} \, d\mu \right)^{1/p_0}.$$
We then have to integrate along $t\in(\eps,\infty)$ and we split the integral into two parts, depending whether $t <4^{j}r^2$ or $t>4^{j}r^2$. We then obtain that 
\begin{align*}
& \left(\aver{B_{x,\eps}} \left|  \int_\eps^\infty T_t (I-e^{-r^2 L})^{N} b\,\frac{dt}{t} \right|^{q_0}\,d\mu \right)^{1/q_0} \\
& \qquad \lesssim \left[\int_{\eps}^{4^j r^2} 2^{-j(\nu+1)} \left(\frac{t}{r^2}\right)^{1/4} \, \frac{dt}{t} + \int_{4^j r^2}^\infty \left(\frac{r^2}{t}\right)^{N/2} \, \frac{dt}{t} \right] \left( \aver{B} |b|^{p_0} \, d\mu \right)^{1/p_0} \\
& \qquad \lesssim 2^{-j(\nu+1/2)}\left( \aver{B} |b|^{p_0} \, d\mu \right)^{1/p_0},
\end{align*}
which corresponds to the desired estimate \eqref{eq:OD-est-extrapol} with $c(j)=2^{-j(\nu+1/2)}$.

Consider now the case $\eps\leq r^2$.  Let again $x \in S_j(B)$ for $j \geq 3$. 
We split the corresponding part of $T^{\#}(I-e^{-r^2L})^{N} b(x)$ into 
\begin{align} \label{eq:decomp-eps-r} \nonumber
	& \sup_{\eps<r^2} \left(\aver{B_{x,\eps}} |T (I-e^{-r^2L})^{N} b|^{q_0}\,d\mu \right)^{1/q_0}\\
	& \qquad +
	\sup_{\eps<r^2} \left(\aver{B_{x,\eps}} | \int_0^\eps T_t (I-e^{-r^2L})^{N} b\,\frac{dt}{\sqrt{t}}|^{q_0}\,d\mu \right)^{1/q_0}
	=: I_1(x) + I_2(x). 
\end{align}
Denote by $\tilde{S}_j(B)$ the slightly enlarged annulus such that $B_{x,\eps} \subseteq \tilde{S}_j(B)$ for $x \in S_j(B)$.
We estimate the first term $I_1(x)$ in \eqref{eq:decomp-eps-r} against the maximal function, localised in $\tilde{S}_j(B)$ due to the restriction of the supremum to small $\eps$ and the assumption $j \geq 3$. This gives for $x \in S_j(B)$
\begin{align*}
	I_1(x)	
%	& = \sup_{\eps<r^2} \left(\aver{Q_{x,\eps}} |\int_0^\infty m(t) Q_t (I-e^{-r^2L})^N b(y)\,\frac{dt}{t}|^{q_0}\right)^{1/q_0}\\
	 \lesssim \calM_{q_0} (\Eins_{\tilde{S}_j(B)} T(I-e^{-r^2L})^{N} b)(x).
\end{align*}
By H\"older's inequality and Kolmogorov's lemma (see, e.g., \cite[Lemma 5.16]{Duo}) for $\calM_{q_0}$, we have 
\begin{align*}
	|2^{j+1}B|^{-1/2} \|I_1\|_{L^2(S_j(B))}
	& \lesssim |2^{j+1}B|^{-1/2} \|\calM_{q_0} (\Eins_{\tilde{S}_j(B)} T(I-e^{-r^2L})^{N} b)\|_{L^2(2^jB)}\\
	& \lesssim |2^{j+1}B|^{-1/q_0} \|T(I-e^{-r^2 L})^{N} b\|_{L^{q_0}(\tilde{S}_j(B))}.
\end{align*}
By Corollary \ref{cor}, we know that $T(I-e^{-r^2L})^{N}$ satisfies $L^{p_0}$-$L^{q_0}$ (strictly) off-diagonal estimates at the scale $r$, thus giving \eqref{eq:OD-est-extrapol} for this part with coefficients $c(j)=2^{-j(\nu+1)}$. 

For $I_2$, on the other hand, we can directly estimate (using Lemma \ref{lemma} - Part 2)
\begin{align*}
	|B_{x,\eps}|^{-1/q_0} \| T_t (I-e^{-r^2L})^{N} b\|_{L^{q_0}(B_{x,\eps})} 
	& \lesssim 2^{-j(\nu+1)} \left(\frac{t}{r^2}\right)^{1/2}  \left( \aver{B} \left|b \right|^{p_0} \, d\mu \right)^{1/p_0}.
%	
%	\left(1+\frac{d(B_{x,\eps},B)^2}{t}\right)^{-\gamma} t^{-\frac{n}{2}(\frac{1}{p_0}-\frac{1}{q_0})} \|b\|_{L^{p_0}(B)}\\
%	& \lesssim \left(\frac{t}{(2^j r)^2}\right)^{\gamma}  t^{-\frac{n}{2}(\frac{1}{p_0}-\frac{1}{q_0})},
\end{align*}
Therefore, we may then integrate in $t\in(0,\eps)$.
By taking the supremum over $\eps\in(0,r^2)$ and over $x$, and together with Minkowski's inequality, we obtain \eqref{eq:OD-est-extrapol} also for $I_2$ with coefficients $c(j)=2^{-j(\nu+1)}$.
\end{proof}

\section{Boundedness of the maximal operator by sparse operators}

As done in previous works (see for example \cite{Petermichl, Hytonen, Lerner1, Lerner2, Lerner3, Lacey}), the analysis will involve a discrete stopping-time argument which relies on nice properties associated with a dyadic structure, which is by now well-known in the context of doubling space. We first recall the main results and then by using this structure we detail the stopping-time argument to bound the maximal operator $T^{\#}$ by some specific operators, called `{\it sparse operators}'.

\subsection{Preliminaries and reminder on dyadic analysis}

We first recall several results about the construction of adjacent dyadic systems (see \cite{Christ, SaWh, Hytonen-Kairema} for more details).

\begin{definition} Let us fix some constants $0<c_0\leq C_0 <\infty$ and $\delta \in (0,1)$. A \emph{dyadic system} (of parameters $c_0,C_0,\delta$) is a family of open subsets $(Q^\ell_\alpha)_{\alpha\in \mathscr{A}_\ell,\, \ell \in \Z}$ satisfying the following properties:
\begin{itemize}
\item For every $\ell \in \Z$, the ambient space $M$ is covered (up to a set with vanishing measure) by the disjoint union of the subsets at scale $\ell$: there exists $Z_\ell$ with $\mu(Z_\ell)=0$ such that
$$ M= \bigsqcup_{\alpha\in \mathscr{A}_\ell} Q^\ell_\alpha \bigsqcup Z_\ell;$$
\item If $\ell \geq k$, $\alpha\in \mathscr{A}_k$ and $\beta \in \mathscr{A}_\ell$ then either $Q^{\ell}_{\beta}\subseteq Q^k_{\alpha}$ or $Q^k_{\alpha}\cap Q^{\ell}_{\beta}=\emptyset$;
\item For every $\ell \in \Z$ and $\alpha \in \mathscr{A}_\ell$, there exists a point $z^\ell_\alpha$ with
\begin{equation}\label{eq:contain}
  B(z^\ell_{\alpha},c_0\delta^\ell)\subseteq Q^\ell_{\alpha}\subseteq B(z^\ell_{\alpha},C_0\delta^\ell)=:B(Q^\ell_{\alpha}).
\end{equation}
\end{itemize}
For a cube $Q^k_\alpha$, $k \in \Z$, $\alpha \in \mathscr{A}_k$, we call the unique cube $Q^{k-1}_\beta$, $\beta \in \mathscr{A}_{k-1}$ for which $Q^k_\alpha \subseteq Q^{k-1}_\beta$, the \emph{parent} of $Q^k_\alpha$. We denote for $Q \in \calD$ the parent by $Q^a$, and call $Q$ \emph{child} of $Q^a$. 
\end{definition}

We refer the reader to \cite{Hytonen-Kairema} for some variant where the negligible $Z_\ell$ does not appear if the subsets are non-necessarily assumed to be open. We also refer to a very recent survey by Lerner and Nazarov \cite{LN} about dyadic structures and how they are used for proving weighted estimates of singular operators.

Then we have the following result (see \cite{Hytonen-Kairema} and references therein):

\begin{theorem}\label{thm:adjacentsystems}
There exist constants $c_0,C_0,\delta$, finite constants $K=K(c_0, C_0, \delta)$ and $\rho=\rho(c_0,C_0,\delta)$ as well as a finite collection of families $\calD^b$, $b = 1, 2, \ldots ,K$, where each $\calD^b$ is a dyadic system (of parameters $c_0,C_0,\delta$) with the following extra property: for every ball $B=B(x,r)\subseteq M$, there exists $b \in \{1,\ldots,K\}$ and $Q\in \calD^b$ with
\begin{equation}\label{property:ball;included}
B\subseteq Q\text{ and } \diam (Q)\leq \rho r.
\end{equation}
We denote 
$$ \calD:= \bigcup_{b=1}^K \calD^b,$$ and call a cube $Q$ {\it dyadic cube} whenever $Q \in \calD$.
\end{theorem}

For every dyadic set $Q\in\calD$, we denote  $\ell(Q):=\delta^k$ where the integer $k$ is determined by
$$ \delta^{k+1} \leq \diam(Q) <\delta^{k}.$$

This last result means that in usual situations, it is sufficient to consider a dyadic system instead of the whole collection of balls.

\begin{definition} Given $\calD^k$ one of the previous dyadic systems and a non-negative weight $h \in L^1_{\loc}$, we define its corresponding maximal operator, weighted by $h$, by
$$ \calM_h^{\calD^k} [f](x):= \sup_{x\in Q\in \calD^k} \ \left(\frac{1}{h(Q)}\int_{Q} |f|\, h d\mu \right), \qquad x \in M,$$
for every $f \in L^1_{\loc}(hd\mu)$.
\end{definition}

\begin{lemma} \label{lem:max} Uniformly in $k\in\{1,\ldots,K\}$ and in the weight $h$, the maximal operator $\calM_h^{\calD^k}$ is of weak type $(1,1)$ and strong type $(p,p)$ for the measure $hd\mu$ and every $p\in(1,\infty]$.
\end{lemma}

We refer the reader to \cite[Theorem 15.1]{LN} for a detailed proof of such results and more details. For completeness, we quickly give the proof.

\begin{proof} Since $\calM_h^{\mcD_k}$ is $L^\infty$-bounded (and so $L^\infty(h d\mu)$-bounded), it suffices by interpolation to check its weak $L^1(h d\mu)$-boundedness.

Fix a function $f\in L^1(h d\mu)$. For every $\lambda>0$, we consider the set
$$ \Omega_\lambda:=\{x\in M, \, \calM_h^{\mcD^k}[f](x) >\lambda\}.$$
Due to the properties of the dyadic system, there exists a collection $\mcQ:=(P)_{P\in \mcQ} \subset \mcD^k$ of dyadic sets such that $\Omega_\lambda = \bigcup_{P\in \mcQ} P$ (up to a subset of measure zero) and such that each $P\in \mcQ$ is maximal in $\Omega_\lambda$ and for every $P\in \mcQ$
$$ \frac{1}{h(P)} \int_{P} |f|\, h d\mu >\lambda.$$
Due to the maximality, the dyadic sets $P\in \mcQ$ are pairwise disjoint and so we conclude that
\begin{align*} 
h(\Omega_\lambda) & = \sum_{P\in \mcQ} h(P) \leq \lambda^{-1} \sum_{P\in \mcQ} \int_{P} |f| \, h d\mu \\
& \leq \lambda^{-1} \|f\|_{L^1(h d\mu)},
\end{align*}
which concludes to the weak $L^1(h d\mu)$-boundedness, uniformly with respect to $h$.
\end{proof}

We will also need the weak-type of a slight modification of the previous maximal function.

\begin{lemma} \label{lem:max-bis}  Fix $k\in\{1,..,K\}$ and consider the maximal function
$$ \calM^*[f](x) :=\sup_{x\in Q\in \calD^k} \ \inf_{y\in Q} \calM[f](y), \qquad x \in M,$$
for every $f \in L^1_{\loc}(hd\mu)$. It follows that $\calM^*[f]=\calM[f]$ almost everywhere.
Consequently, the maximal operator $\calM^*$ is of weak type $(1,1)$ and strong type $(p,p)$ for every $p\in(1,\infty]$.
\end{lemma}

\begin{proof} Indeed, since the quantity 
$\inf_{y\in Q} \calM{f}(y)$ is decreasing with respect to $Q$, it follows that 
$$ \calM^*[f](x) = \lim_{\genfrac{}{}{0pt}{}{x\in Q}{\diam(Q) \to 0}} \inf_{y\in Q} \calM[f](y) = \calM[f](x),$$
where we used the Lebesgue differentiation lemma which implies the last equality for almost every $x\in M$.
\end{proof}

\subsection{Upper estimates of the maximal operator with sparse operators} 
From the previous subsection we know that we have several dyadic grids $\mcD^b$ for $b\in \{1,\ldots,K\}$.
In the sequel, we denote $\mcD:=\bigcup_{b=1}^K \mcD^b$ and call {\it dyadic set} any element of $\mcD$.

\begin{definition}[Sparse collection] A collection of dyadic sets $\mcS:=(P)_{P\in \mcS} \subset \mcD$ is said to be {\it sparse} if for each $P\in \mcS$ one has
\begin{equation}
\sum_{Q\in \ch_{\mcS}(P)} \mu(Q) \leq \frac{1}{2} \mu(P), \label{eq:sparse} \end{equation}
where $\ch_{\mcS}(P)$ is the collection of $\mcS$-children of $P$: the maximal elements of $\mcS$ that are strictly contained in $P$. 
\end{definition}

For a dyadic cube $Q\in \calD$, we denote by $5Q$ its neighbourhood
$$ 5Q:=\{x\in M,\ d(x,Q)\leq 4 \ell(Q)\}.$$

\begin{theorem} \label{thm:sparse} Consider an exponent $p \in (p_0,q_0)$. 
There exists a constant $C>0$ such that for all $f \in L^p$ and $g\in L^{p'}$  both supported in  $5Q_0$ for some $Q_0 \in \calD$, there exists a sparse collection $\calS\subset \calD$ (depending on $f,g$) with
\begin{align*}
	\left| \int_{Q_0} Tf \cdot g\, d\mu \right| 
	\leq C \sum_{P \in \calS} \mu(P) \left(\aver{5P} |f|^{p_0}\, d\mu\right)^{1/p_0} \left(\aver{5P} |g|^{q_0'}\, d\mu\right)^{1/q_0'}.
\end{align*}
\end{theorem}

\begin{proof}
Let $p \in (p_0,q_0)$. Suppose $f \in L^p$ and $g\in L^{p'}$, supported in $5Q_0$ for a dyadic set $Q_0 \in \calD$. 
Fix the parameter $b\in \{1,...,K\}$ such that $Q_0\in \calD^b$.
For some large enough constant $\eta$ (which will be fixed  later), define the subset
$$
	E=\left\{x\in Q_0 \,:\, \max \Big(\calM^*_{Q_0,p_0}f(x),T^{\#}_{Q_0}f(x) \Big) > \eta \left(\aver{5 Q_0}|f|^{p_0}\, d\mu \right)^{1/p_0} \right\},
$$
where both $\calM^*_{Q_0,p_0}$ and $T^{\#}_{Q_0}$ are defined relative to the initial subset $Q_0\in \calD^b$ as follows: for every $x\in Q_0$
$$ \calM^*_{Q_0,p_0}[f](x):= \sup_{\genfrac{}{}{0pt}{}{x\in Q\subset Q_0}{Q\in \calD^b}}  \ \inf_{y\in Q} \calM_{p_0}[f](y)$$
and
$$
	T^{\#}_{Q_0}f(x) =  \sup_{\genfrac{}{}{0pt}{}{x\in Q\subset Q_0}{Q\in \calD^b}}	\left(\aver{Q} \Big| T \left[ \int_{\ell(Q)^2}^\infty  Q_t^{(N)}(f) \, \frac{dt}{t} \right] \Big|^{q_0}\,d\mu\right)^{1/q_0}. 
$$
We extend both $\calM^*_{Q_0,p_0}$ and $T^{\#}_{q_0}$ by $0$  outside $Q_0$.

Due to the property of dyadic subsets, we know that every $Q\in \calD^b$ is contained in a ball with radius equivalent to $\ell(Q)$. So up to some implicit constants, $\calM^*_{Q_0,p_0}$ is bounded by the Hardy-Littlewood maximal function $\calM_{p_0}$ (see Lemma \ref{lem:max-bis}) and $T^{\#}_{Q_0}$ is controlled by the maximal operator $T^{\#}$. So Proposition \ref{prop:op-max} yields that both $\calM^*_{Q_0,p_0}$ and $T^{\#}_{Q_0}$ are of weak type $(p_0,p_0)$.

Then it follows that $\mu(E) \lesssim \frac{1}{\eta} \mu(Q_0)$. So if $\eta$ is chosen large enough, then we know that $E$ is an open proper subset of $Q_0$.  
In the sequel, all the implicit constants will only depend on the ambient space. For convenience, we only emphasise the dependence relatively to $\eta$, which will be useful later to show how $\eta$ can be fixed.

Consider a maximal dyadic covering of $E$, which is a collection of dyadic subsets $(B_j)_j\subset \calD^b$ such that
\begin{itemize}
\item The collection covers $E$: $E= \bigsqcup_j B_j$, up to a set of null measure with disjointness of the dyadic cubes;
\item The dyadic cubes are maximal, in the sense that for every $j$, $B_j^{a} \cap E^c \neq \emptyset$, where we recall that $B_j^a$ is the parent of $B_j$.
\end{itemize}

Since $\mu(B_j)\leq \mu(E)\lesssim \eta^{-1} \mu(Q_0)$, if $\eta$ is chosen large enough and using the doubling property of the measure $\mu$, we deduce that we also have
$$ \mu(B_j^a) \leq \mu(Q_0).$$
Due to the properties of the dyadic system, we then deduce that $B_j^a$ is included in $Q_0$, and so the maximality of $B_j$ yields
\begin{equation} 
\max\left\{ \inf_{y\in B_j^a} \calM_{p_0}[f](y) ,\  \left(\aver{B_j^a} \Big| T \left[ \int_{\ell(B_j^a)^2}^\infty  Q_t^{(N)}(f) \, \frac{dt}{t} \right] \Big|^{q_0}\,d\mu\right)^{1/q_0} \right\} \leq  \eta \left(\aver{5 Q_0}|f|^{p_0}\, d\mu \right)^{1/p_0}.
\label{eq:max}
\end{equation}

We first initialize the collection $\calS:=\{Q_0\}$ and we are going to  build it in a recursive way. For $B \in \calD$, define the operator $T_{B}$ by
\begin{align*}
	T_{B} f:= T \left[ \int_0^{\ell(B)^2}  Q_t^{(N)}( f \Eins_{5B} ) \, \frac{dt}{t} \right].
\end{align*}

\medskip
\noindent
{\bf Step 1:} In this step, we first aim to show that for some numerical constant $C_0$, 
\begin{align}
	\left| \int_{Q_0} Tf \cdot g\, d\mu \right| 
	\leq C_0\eta  |Q_0| \left(\aver{Q_0}|f|^{p_0}\, d\mu \right)^{1/p_0}  \left(\aver{Q_0}|g|^{q_0'}\, d\mu\right)^{1/q_0'}+ \sum_{j} \left|\int_{B_j} T_{B_j}f \cdot g \, d\mu \right|. \label{eq:amon11}
\end{align}
Aiming that, write
\begin{align*}
	\left| \int_{Q_0} Tf \cdot g\, d\mu \right| \leq \left|\int_{Q_0\setminus E} Tf \cdot g\, d\mu \right| + \left|\int_{E} Tf \cdot g\, d\mu \right|.
\end{align*}
For the first part, notice that $|Tf(x)| \leq T^{\#}_{Q_0}f(x)\leq \eta \left(\aver{5Q_0}|f|^{p_0}\, d\mu\right)^{1/p_0}$ for a.e. $x \in Q_0\setminus E$ by definition of $E$. Hence
\begin{align*}	
	 \left|\int_{Q_0\setminus E} Tf \cdot g \, d\mu \right| 
	 \leq \eta \mu(Q_0) \left(\aver{5Q_0}|f|^{p_0}\, d\mu\right)^{1/p_0} 
	 \left(\aver{Q_0}|g|^{q_0'}\, d\mu\right)^{1/q_0'}.
\end{align*}
For the part on $E$, we use the covering to have
\begin{align*}	
	\left|\int_{E} T f \cdot g \, d\mu \right|
	 &\leq \sum_{j} \left|\int_{B_j} T_{B_j} f \cdot g\, d\mu \right|
	 + \left|\sum_{j}  \int_{B_j} (T-T_{B_j}) f \cdot g \, d\mu \right|\\
	 & \leq  \sum_{j} \left|\int_{B_j} T_{B_j} f \cdot g \, d\mu \right|
	 + \sum_{j}  \mu(B_j) \left(\aver{B_j}|(T-T_{B_j})f|^{q_0}\, d\mu\right)^{1/q_0} \left(\aver{B_j}|g|^{q_0'}\, d\mu\right)^{1/q_0'}.
\end{align*}
The first sum enters into the recursion and is acceptable in view of \eqref{eq:amon11}. For the second sum, we have
\begin{align}
	\left|(T-T_{B_j})f\right| 
	\leq  \left|T\left[\int_{\ell(B_j)^2}^\infty Q_t^{(N)}(f)  \,\frac{dt}{t} \right]\right|
	+ \left|T \left[\int_0^{\ell(B_j)^2} Q_t^{(N)} (\Eins_{(5B_j)^c} f) \,\frac{dt}{t} \right]\right|. \label{eq:split2}
\end{align}
Using the doubling property, we can estimate the first term against the maximal operator, and get
\begin{align*} %\label{eq:eq-max} 
	\left(\aver{B_j}\Big| T \left[\int_{\ell(B_j)^2}^\infty Q_t^{(N)} f \,\frac{dt}{t}\right] \Big|^{q_0}\, d\mu \right)^{1/q_0} 
	&\lesssim  \left(\aver{B_j^a}\Big| T \left[\int_{\ell(B_j)^2}^\infty Q_t^{(N)} f \,\frac{dt}{t} \right] \Big|^{q_0}\right)^{1/q_0}\\	
	& \lesssim \inf_{z \in B_j^a} T^{\#}f(z) +  \left(\aver{B_j^a}\Big| T\left[\int_{\ell(B_j)^2}^{\ell(B_j^a)^2} Q_t^{(N)}f \,\frac{dt}{t}\right] \Big|^{q_0}\right)^{1/q_0}.
\end{align*}	
By the maximality of the dyadic cubes $B_j$, we know that $B_j^a$ intersects $E^c$ hence from \eqref{eq:max}, we have
$$ \inf_{z \in B_j^a} T^{\#}f(z) \leq \eta \left(\aver{5Q_0}|f|^{p_0}\, d\mu\right)^{1/p_0}.$$
Moreover, we also know that for every dyadic set $B_j$ we have
$$ \inf_{y\in B_j^a} \calM_{p_0}[f](y) \leq \eta \left(\aver{5Q_0}|f|^{p_0}\, d\mu\right)^{1/p_0},$$
which yields in particular that
\begin{align*}
\left(\aver{B_j^a}\Big| T \left[\int_{\ell(B_j)^2}^{\ell(B_j^a)^2} Q_t^{(N)}f \,\frac{dt}{t} \right] \Big|^{q_0}\right)^{1/q_0} & \lesssim \int_{\ell(B_j)^2}^{\ell(B_j^a)^2} \left(\aver{B_j^a} \Big| T Q_t^{(N)}f \Big|^{q_0} \, d\mu \right)^{1/q_0} \, \frac{dt}{t} \\
& \lesssim \eta \left(\aver{5Q_0}|f|^{p_0}\, d\mu\right)^{1/p_0}.
\end{align*}
We do not detail this last inequality, since it is a simpler and a particular case of the next one.	
	
For the second term in \eqref{eq:split2}, we use $L^{p_0}$-$L^{q_0}$ off-diagonal estimates for $T_t=TQ_t^{(N)}$ from Assumption \ref{assum-T} (b).
 We have that
$$ (5B_j)^c \subset \bigcup_{k=2}^\infty S_k(B_j),$$
and can therefore decompose
\begin{align*}
 \left(\aver{B_j}\Big| T \left[\int_0^{\ell(B_j)^2} Q_t^{(N)} (f\Eins_{(5B_j)^c})\,\frac{dt}{t} \right] \Big|^{q_0}\, d\mu\right)^{1/q_0} & \leq \int_0^{\ell(B_j)^2} \left(\aver{B_j}\Big| T_t (f\Eins_{(5B_j)^c})  \Big|^{q_0}\, d\mu\right)^{1/q_0} \,\frac{dt}{t} \\
 & \leq \sum_{k\geq 2} \int_0^{\ell(B_j)^2} \left(\aver{B_j}\Big| T_t (f\Eins_{S_k(B_j)}) \Big|^{q_0}\, d\mu\right)^{1/q_0} \,\frac{dt}{t}. 
\end{align*}
For fixed $t\in (0,\ell(B_j)^2)$ we know that $T_t$ satisfies $L^{p_0}$-$L^{q_0}$ off-diagonal estimates at the scale $\sqrt{t}$. We then cover $S_k(B_j)$ by 
balls of radius $\sqrt{t}$, with a finite overlap property (by the doubling property of the measure).
We then deduce that these balls $R$ satisfy that
$$ d(R,B_j) \geq \ell(B_j) \qquad \textrm{and} \quad d(R,B_j) \simeq d(S_k(B_j), B_j) \simeq 2^k \ell(B_j).$$ 
Moreover, the number of these balls needed to cover $S_k(B_j)$ is controlled by
\begin{equation} \#\{R\} \lesssim \left(\frac{2^k \ell(B_j)}{\sqrt{t}}\right)^\nu. \label{eq:nombre} \end{equation}
By summing over such a covering, we get
\begin{align*}
 \left(\aver{B_j}\Big|T_t (f\Eins_{S_k(B_j)}) \Big|^{q_0}\, d\mu\right)^{1/q_0} 
 & \lesssim \sum_{R} \left(1+\frac{d(R,B_j)^2}{t}\right)^{-\frac{\nu+1}{2}}  \left(\aver{R}| f|^{p_0}\, d\mu\right)^{1/p_0} \\
 & \lesssim \left(1+\frac{4^k \ell(B_j)^2}{t}\right)^{-\frac{\nu+1}{2}} \left(\frac{2^k \ell(B_j)}{\sqrt{t}}\right)^{\nu/p_0} |2^kB_j|^{-1/p_0}  \sum_{R}  \left(\int_{R}| f|^{p_0}\, d\mu\right)^{1/p_0}.
\end{align*}
By H\"older's inequality with the bounded overlap property of the collection $\{R\}$ with \eqref{eq:nombre}, we then have
$$ \sum_{R}  \left(\int_{R}| f|^{p_0}\, d\mu\right)^{1/p_0} \lesssim \left( \int_{S_k(B_j)} |f|^{p_0} \, d\mu \right)^{1/p_0} \left(\frac{2^k \ell(B_j)}{\sqrt{t}}\right)^{\nu/p_0'},$$
hence
\begin{align*}
  \left(\aver{B_j}\Big|T_t (f\Eins_{S_k(B_j)}) \Big|^{q_0}\, d\mu\right)^{1/q_0} &  \lesssim \left(1+\frac{4^k \ell(B_j)^2}{t}\right)^{-\frac{\nu+1}{2}} \left(\frac{2^k \ell(B_j)}{\sqrt{t}}\right)^{\nu}  \left(\aver{S_k(B_j)} |f|^{p_0} \, d\mu \right)^{1/p_0} \\
 & \lesssim \left(\frac{\sqrt{t}}{2^k \ell(B_j)}\right)  \left(\aver{S_k(B_j)} |f|^{p_0} \, d\mu \right)^{1/p_0}.
\end{align*}
We therefore get
\begin{align}  \nonumber 
	&\left(\aver{B_j}\Big| T \left[\int_0^{\ell(B_j)^2} Q_t^{(N)} (f\Eins_{(5B_j)^c})\,\frac{dt}{t} \right] \Big|^{q_0}\, d\mu\right)^{1/q_0}\\ \nonumber 
	&\qquad \lesssim \sum_{k=2}^\infty   \left(\aver{S_k(B_j)} |f|^{p_0} \, d\mu \right)^{1/p_0} \int_0^{\ell(B_j)^2}  \left(\frac{\sqrt{t}}{2^k \ell(B_j)}\right) \,\frac{dt}{t} \\ \nonumber 
	& \qquad \lesssim  \sum_{k=2}^\infty 2^{-k}  \left(\aver{S_k(B_j)} |f|^{p_0} \, d\mu \right)^{1/p_0}
	 \lesssim \sup_{k \geq 2} \left(\aver{2^k B_j}|f|^{p_0}\, d\mu \right)^{1/p_0}\\
	 & \label{eq:eq-OD} \qquad \lesssim \inf_{z \in B_j^a} \calM_{p_0}f(z)
	 \lesssim \eta \left(\aver{5Q_0}|f|^{p_0}\right)^{1/p_0},
\end{align}
where we used \eqref{eq:max}.

On the other hand, 
\begin{align*}
	\left(\aver{B_j}|g|^{q_0'}\, d\mu \right)^{1/q_0'} \leq \inf_{z \in B_j} \calM_{q_0'} g(z),
\end{align*}
and, using $\bigcup_{j} B_j=E$, Kolmogorov's inequality, the fact that $\mu(E) \lesssim \mu(Q_0)$ (since $\eta$ will be chosen larger than $1$) and $\supp g \subseteq 5Q_0$,
\begin{align*}
	\sum_{j} \mu(B_j) \inf_{z \in B_j} \calM_{q_0'}[g](z)
	&  \leq
	\int_{E} \calM_{q_0'}[g](z)\,dz 
	\lesssim \mu(E)^{1-1/q_0'} \||g|^{q_0'}\|_1^{1/q_0'}	\\
	& \lesssim  \mu(Q_0) \left( \aver{5Q_0} |g|^{q_0'} \, d\mu \right)^{1/q_0'}.
\end{align*}

Therefore, putting all the estimates together, we have shown that 
\begin{align*}
	&\sum_{j}  \mu(B_j) \left(\aver{B_j}|(T-T_{B_j})f|^{q_0}\, d\mu\right)^{1/q_0} \left(\aver{B_j}|g|^{q_0'}\, d\mu\right)^{1/q_0'} \\
	& \qquad \qquad  \lesssim \eta \mu(Q_0)  \left(\aver{5Q_0}|f|^{p_0}\right)^{1/p_0} \left(\aver{5Q_0}|g|^{q_0'}\right)^{1/q_0'},
\end{align*}
where the implicit constant only depends on the ambient space through previous numerical constants. This concludes the proof of \eqref{eq:amon11}.

\medskip
\noindent
{\bf Step 2:} Recursion and conclusion.

Starting from the initial dyadic cube $Q_0$, we have built a collection of dyadic cubes $(Q_1^j)_j$ such that
\begin{align*}
	\left| \int_{Q_0} Tf \cdot g\, d\mu \right| 
	\leq C_0\eta  \mu(Q_0) \left(\aver{5Q_0}|f|^{p_0}\, d\mu \right)^{1/p_0}  \left(\aver{5Q_0}|g|^{q_0'}\, d\mu\right)^{1/q_0'}+ \sum_{j} \left|\int_{Q_1^j} T_{Q_1^j}f^j \cdot g^j\, d\mu \right|,
\end{align*}
where $f^j$ and $g^j$ are both supported in $5Q_1^j$ and are respectively pointwise bounded by $f$ and $g$. 
Moreover, the following properties hold:
\begin{enumerate}
\item Small measure: for some numerical constant $\tilde K$
$$ \sum_j \mu(Q_1^j) \leq \frac{\tilde K}{\eta} \mu(Q_0)$$
\item disjointness and covering: $(Q_1^j)_j$ are pairwise disjoint and included in $Q_0$.
\end{enumerate}
We then add all these cubes to the collection $\calS$, $\calS=\calS \cup \bigcup_j\{Q_1^j\}$.
And we iterate the procedure. For every cube $Q_1^j$, there exists 
a collection of dyadic cubes $(Q_2^{j,k})_k$ such that
\begin{align*}
	\left| \int_{Q_1^j} Tf^j \cdot g^j\, d\mu \right| 
	\leq C_0\eta  \mu(Q_1^j) \left(\aver{5Q_1^j}|f|^{p_0}\, d\mu \right)^{1/p_0}  \left(\aver{5Q_1^j}|g|^{q_0'}\, d\mu\right)^{1/q_0'}+ \sum_{k} \left|\int_{Q_2^{j,k}} T_{Q_2^{j,k}}f^{j,k} \cdot g^{j,k}\, d\mu \right|,
\end{align*}
with the properties that $f^{j,k}$ and $g^{j,k}$ are pointwisely bounded by $f$ and $g$ with also
\begin{enumerate}
\item Small measure
$$ \sum_k \mu(Q_2^{j,k}) \leq \frac{\tilde K}{\eta} \mu(Q_1^j)$$
\item disjointness and covering: $(Q_2^{j,k})_k$ are pairwise disjoint and included in $Q_1^j$.
\end{enumerate}
We then add all these cubes to the collection $\calS$, $\calS=\calS \cup \bigcup_{j,k} \{Q_2^{j,k}\}$.
We iterate this reasoning which allows us to build the collection $\calS$ with the property that
\begin{align*}
	\left| \int_{Q_0} Tf \cdot g\, d\mu \right| 
	\leq C_0\eta \sum_{Q \in \calS} \mu(Q) \left(\aver{5Q}|f|^{p_0}\, d\mu \right)^{1/p_0}  \left(\aver{5Q}|g|^{q_0'}\, d\mu\right)^{1/q_0'}.
\end{align*}
Indeed, it is easy to check that the remainder term at the $i^{th}$-step is an integral over a subset of measure which tends to $0$ as $i$ goes to $\infty$. So for fixed $f\in L^p$ and $g \in L^{p'}$ with $p'<\infty$, the remainder term also tends to $0$.

It remains us to check that this collection $\calS$ is sparse.

So consider $Q\in \calS$. By the disjointness property of the selected dyadic cubes, it is clear that any child $\bar Q \in ch_{\calS}(Q)$ has been selected (strictly) after $Q$ and in the collection $\calS_Q$ {\it generated} by $Q$.
Using the smallness property of the measure in the algorithm, we know that summing over all the cubes $R$ selected, stricly after $Q$, in the collection generated by $Q$ gives us
\begin{align*}
 \sum_{R \in \calS_Q} \mu(R) & = \sum_{\ell \geq 1} \left(\frac{\tilde K}{\eta}\right)^\ell \mu(Q) \\
&  \leq \frac{\tilde K}{\eta-\tilde K} \mu(Q).
\end{align*}
We then deduce that by choosing $\eta$ large enough, the selected collection is sparse.
\end{proof}

\section{Boundedness of a sparse operator}
\label{sec:weights}

\begin{definition}[$A_p$ weight] A measurable function $\omega:M \rightarrow (0,\infty)$ is an $A_p$ weight for some $p\in(1,\infty)$ if
$$ [\omega]_{A_p}:= \sup_{\textrm{ball\ } B}  \left( \aver{B} \omega \, d\mu \right) \left( \aver{B} \omega^{1-p'}\, d\mu \right)^{p-1} <\infty,$$
with $p'$ the conjugate exponent $p'=p/(p-1)$. For $p=1$, we extend this notion with the characteristic constant
$$ [\omega]_{A_1}:= \sup_{\textrm{ball\ } B} \left(\aver{B} \omega\, d\mu \right) \left(\essinf_{x\in B} \omega(x) \right)^{-1}.$$
\end{definition}

\begin{definition}[$RH_q$ weight] A measurable function $\omega:M \rightarrow (0,\infty)$ is a $RH_{q}$ weight for some $q\in(1,\infty)$ if
$$ [\omega]_{RH_q}:= \sup_{\textrm{ball\ } B} \left(\aver{B} \omega^q \, d\mu \right)^{1/q} \left(\aver{B} \omega\, d\mu \right)^{-1} <\infty.$$
For $q=\infty$, we extend this notion with the characteristic constant
$$ [\omega]_{RH_\infty}:= \sup_{\textrm{ball\ } B}  \left(\esssup_{x\in B} \omega(x) \right)  \left(\aver{B} \omega\, d\mu \right)^{-1}.$$
\end{definition}

We recall some well-known properties on the weight.

\begin{lemma} \label{lem:weight}
\begin{itemize}
\item[(a)] For $p\in(1,\infty)$ and a weight $\omega$, then $\omega\in A_p$ if and only if $\omega^{1-p'} \in A_{p'}$ with
$$ [\omega^{1-p'}]_{A_{p'}} = [\omega]_{A_p}^{p'-1}.$$
\item[(b)] (see \cite{JN}) For $q\in[1,\infty]$, $s\in[1,\infty)$ and a weight $\omega$, then $\omega \in A_q \cap RH_s$ if and only if $\omega^s \in A_{s(q-1)+1}$ with
$$ [\omega^s]_{A_{s(q-1)+1}} \leq [\omega]_{A_q}^s [\omega]_{RH_s}^s.$$
\end{itemize}
\end{lemma}

We prove the following sharp weighted estimates for the ``sparse'' operators.

\begin{proposition} \label{prop:sparse} Let $p_0,q_0 \in[1,\infty]$ be two exponents with $p_0<q_0$, and let $p \in (p_0,q_0)$. Suppose that $S$ is a bounded operator on $L^p$ and that there exists a constant $c>0$ such that for all $f\in L^p$ and $g\in L^{p'}$ there exists a sparse collection $\mcS$ with 
 $$ \left| \langle S(f), g\rangle \right| \leq c \sum_{P\in \mcS} \left(\aver{5P} |f|^{p_0}\, d\mu \right)^{1/p_0} \left(\aver{5P} |g|^{q_0'}\, d\mu \right)^{1/q_0'} \mu(P).$$
Denote 
$$ r:=\left(\frac{q_0}{p}\right)' \left(\frac{p}{p_0}-1\right)+1 \qquad \textrm{and} \qquad \delta:= \min\{q_0',p_0(r-1)\}.$$
Then there exists a constant $C=C(S,p,p_0,q_0)$ such that for every weight $\omega \in A_{\frac{p}{p_0}} \cap RH_{\left(\frac{q_0}{p}\right)'}$, the operator $S$ is bounded on $L^p_\omega$ with
$$ \| S \|_{L^p_\omega \to L^p_\omega} \leq C  \left([\omega]_{A_{\frac{p}{p_0}}} [\omega]_{ RH_{\left(\frac{q_0}{p}\right)'}}\right)^\alpha,$$
with 
$$ \alpha:=\frac{1}{\delta} \left(\frac{q_0}{p}\right)'=\max\left\{\frac{1}{p-p_0}, \frac{q_0-1}{q_0-p}\right\}.$$
In particular, by defining the specific exponent
$$ \mathfrak{p}:=1+\frac{p_0}{q_0'} \in (p_0,q_0),$$
we have $\alpha=\frac{1}{p-p_0}$ if $p \in(p_0, \mathfrak{p}]$ and $\alpha=\frac{q_0-1}{q_0-p}$ if $p \in [\mathfrak{p},q_0)$.
\end{proposition}

\begin{rem}
The property that $p_0<\mathfrak{p}$ is equivalent to the condition $p_0<q_0$ and the fact that $\mathfrak{p}<q_0$ is also equivalent to the condition $p_0<q_0$. So the assumption guarantees us that
$$ p_0<p<q_0.$$
\end{rem}

We note that using extrapolation theory (as developed in \cite[Theorem 4.9]{AM}) and by tracking the behaviour of implicit constants with respect to the weights, then a sharp weighted estimate for one particular exponent $p\in(p_0,q_0)$ allows us to get the sharp weighted estimates for all the exponents in the range $p\in(p_0,q_0)$. Here, we are going to detail a proof which directly gives the weighted estimates for all such exponents.

\medskip

\begin{rem}
\begin{itemize}
\item In the case where $q_0=p_0'$, it is $\mathfrak{p}=2$ and we obtain sharp weighted estimates with the power
$$ \alpha=\max\{\frac{1}{p-p_0},\frac{1}{p+p_0-pp_0}\}.$$
\item In particular, in the situation where $p_0=1$ and $q_0=\infty$, we re-obtain the ``usual'' sharp behavior, dicted by the $A_2$-conjecture, with the power
$$ \alpha=\max\{1,1/(p-1)\}.$$
\item In the case $q_0=\infty$, we obtain
$$ \alpha= \max\{1,(p-p_0)^{-1} \}$$
which is the same exponent as in \cite{BCDH}, and allows to regain their result (the linear part) as explained in Subsection \ref{subsec:fourier}.
\end{itemize}
\end{rem}

\begin{rem} For a weight $\omega$, we know (see Lemma \ref{lem:weight} and \cite[Lemma 4.4]{AM}) that
$$\omega \in A_{\frac{p}{p_0}} \cap RH_{\left(\frac{q_0}{p}\right)'} \Longleftrightarrow \sigma:= \omega^{1-p'} \in A_{\frac{p'}{q_0'}} \cap RH_{\left(\frac{p_0'}{p'}\right)'}.$$
This is also equivalent to
$$ \omega^{\left(\frac{q_0}{p}\right)'} \in A_r$$
with $r:=\left(\frac{q_0}{p}\right)' \left(\frac{p}{p_0}-1\right)+1.$
We have the following estimate on the characteristic constants
$$ [\omega^{\left(\frac{q_0}{p}\right)'}]_{A_{r}} \lesssim  \left([\omega]_{A_{\frac{p}{p_0}}} [\omega]_{ RH_{\left(\frac{q_0}{p}\right)'}}\right)^{\left(\frac{q_0}{p}\right)'}$$
and
$$ [\sigma]_{A_{\frac{p'}{q_0'}}} [\sigma]_{RH_{\left(\frac{p_0'}{p'}\right)'}} \lesssim \left([\omega]_{A_{\frac{p}{p_0}}} [\omega]_{ RH_{\left(\frac{q_0}{p}\right)'}}\right)^{p'-1}.$$
 \end{rem}

\begin{proof}[Proof of Proposition \ref{prop:sparse}]
Let us define three weights $\sigma:=\omega^{1-p'}$, 
$$ u:= \sigma^{\left(\frac{p_0'}{p'}\right)'} \qquad \textrm{and} \qquad v:=\omega^{\left(\frac{q_0}{p}\right)'}.$$
So that $u=v^{1-r'}$ with 
$$ r:=\left(\frac{q_0}{p}\right)' \left(\frac{p}{p_0}-1\right)+1.$$
Following the previous remark with Lemma \ref{lem:weight}, the fact that $\omega \in A_{\frac{p}{p_0}} \cap RH_{\left(\frac{q_0}{p}\right)'}$ yields that $v\in A_r$ and so
\begin{equation*} \sup_{\textrm{ball\ } B} \left(\aver{B} v \, d\mu \right) \left( \aver{B} u \, d\mu \right)^{r-1} \leq [v]_{A_r} \lesssim [\omega]^{\left(\frac{q_0}{p}\right)'}, \end{equation*}
where we set
$$ [\omega]:=[\omega]_{A_{\frac{p}{p_0}}} [\omega]_{ RH_{\left(\frac{q_0}{p}\right)'}}$$
the characteristic constant of the weight $\omega$ in the class $A_{\frac{p}{p_0}} \cap  RH_{\left(\frac{q_0}{p}\right)'}$.
Using the comparison between dyadic subsets with balls and the doubling property of the measure $\mu$, we then deduce that
\begin{equation} \sup_{Q\in \mcD} \left(\aver{Q} v \, d\mu \right) \left( \aver{Q} u \, d\mu \right)^{r-1} \lesssim [v]_{A_r} \lesssim [\omega]^{\left(\frac{q_0}{p}\right)'}. \label{eq:w} 
\end{equation}

We know that the dual space (with respect to the measure $d\mu$) of $L^p_{\omega}$ is $L^{p'}_{\sigma}$. So the desired $L^p_\omega$-boundedness of $S$ is equivalent to the following inequality:
\begin{equation}
 \left| \langle S(f), g \rangle \right| \lesssim [\omega]^\alpha \|f\|_{L^p_\omega} \|g\|_{L^{p'}_\sigma}.
 \label{eq:amontrer}
\end{equation}

Let us fix two functions $f\in L^p_\omega$ and $g\in L^{p'}_\sigma$. Then by assumption there exists a sparse collection $\mcS$ such that 
$$ \left| \langle S(f), g\rangle \right| \leq c \sum_{P\in \mcS} \left(\aver{5P} |f|^{p_0}\, d\mu \right)^{1/p_0} \left(\aver{5P} |g|^{q_0'}\, d\mu \right)^{1/q_0'} \mu(P).$$

For every $P\in \mcS$, we know that there exists a dyadic cube $\bar P$ such that $5P \subset \bar P$ and $\mu(\bar P)\lesssim \mu(5P)$.
We split $\mcS$ into $K$ collections $(\mcS_k)_{k=1,..,K}$ for which $\bar P\in \mcD^k$. Each collection $\mcS_k$ is still sparse, since it is a sub-collection of $\mcS$.

We now fix $k \in \{1,..,K\}$.
For every $P\in \mcS_k$, we set $E_P\subset P$ the set of all $x\in P$ which are not contained in any $\mcS_k$-child of $P$. By the sparseness property of $\mcS_k$, we then have
$$ \mu(P) \leq 2 \mu(E_P)$$
and the sets $(E_P)_{P\in\mcS_k}$ are pairwise disjoint. 

So we have
\begin{equation} \left| \langle S(f), g\rangle \right| \lesssim \sum_{k=1}^K \sum_{P\in \mcS_k} \left(\aver{\bar P} |f|^{p_0}\, d\mu \right)^{1/p_0} \left(\aver{\bar P} |g|^{q_0'}\, d\mu \right)^{1/q_0'} \mu(E_P). \label{eq:11}
\end{equation}

We then change the measure with the weight $u$ as follows
\begin{align}
  \left(\aver{\bar P} |f|^{p_0}\, d\mu \right)^{1/p_0} & = \left(\aver{\bar P} |u^{-1/p_0} f|^{p_0} \, u d\mu \right)^{1/p_0} \nonumber \\
 \label{eq:eq11}   & = \left(\frac{1}{u(\bar P)} \int_{\bar P} |u^{-1/p_0} f|^{p_0} \, ud\mu \right)^{1/p_0} \left( \aver{\bar P} u\, d\mu \right)^{1/p_0} .
 \end{align}
Similarly, we have
\begin{align}
  \left(\aver{\bar P} |g|^{q_0'}\, d\mu \right)^{1/q_0'} & = \left(\aver{\bar P} |v^{-1/q_0'} g|^{q_0'} v\, d\mu \right)^{1/q_0'} \nonumber \\
 \label{eq:eq12}   & = \left(\frac{1}{v(\bar P)} \int_{\bar P} |v^{-1/q_0'} g|^{q_0'} \, vd\mu \right)^{1/q_0'} \left( \aver{\bar P} v\, d\mu \right)^{1/q_0'}.
 \end{align}

Set $\alpha:=\delta^{-1} \left(\frac{q_0}{p}\right)'$, with $\delta:= \min\{q_0',p_0(r-1)\}$ and $\beta := \frac{1}{p_0} - \frac{r-1}{q_0'}$.
Remark that $\beta\leq 0$ is equivalent to $p\geq \mathfrak{p}$ and is also equivalent to $\delta=q_0'$; whereas $\beta\geq 0$ is equivalent to $p\leq \mathfrak{p}$ and is also equivalent to $\delta=p_0(r-1)$. We are first going to detail the end of the proof in the case $\beta \leq 0$ and then explain that the situation $\beta \geq 0$ is very similar.

\medskip
\noindent
{\bf Step 1: } Case $p \geq \mathfrak{p}$ (i.e. $\beta\leq 0$). \\

Putting the two last estimates \eqref{eq:eq11} and \eqref{eq:eq12} in \eqref{eq:11} yields
\begin{align} \left| \langle S(f), g\rangle \right| 
% & \lesssim \sum_{k=1}^K \sum_{P\in \mcS_k} 
%\left(\frac{1}{u(\bar P)} \int_{\bar P} |u^{-1/p_0} f|^{p_0} \, ud\mu \right)^{1/p_0}\left(\frac{1}{v(\bar P)} %\int_{\bar P} |v^{-1/q_0'} g|^{q_0'} \, vd\mu \right)^{1/q_0'}\left( \aver{\bar P} u\, d\mu \right)^{1/p_0} %\left( \aver{\bar P} v\, d\mu \right)^{1/q_0'} \mu(E_p) \nonumber \\
 & \lesssim [\omega]^\alpha \sum_{k=1}^K \sum_{P\in \mcS_k} 
\left(\frac{1}{u(\bar P)} \int_{\bar P} |u^{-1/p_0} f|^{p_0} \, ud\mu \right)^{1/p_0} \nonumber \\
& \qquad \qquad \left(\frac{1}{v(\bar P)} \int_{\bar P} |v^{-1/q_0'} g|^{q_0'} \, vd\mu \right)^{1/q_0'}  \left( \aver{\bar P} u\, d\mu \right)^{\beta}  \mu(E_P), \label{eq:111}
\end{align}
where we used that
\begin{equation} \left( \aver{\bar P} u\, d\mu \right)^{(r-1)/\delta} \left( \aver{\bar P} v\, d\mu \right)^{1/\delta} \lesssim [\omega]^{\delta^{-1}\left(\frac{q_0}{p}\right)'} \label{eq:eq13} \end{equation}
which comes from \eqref{eq:w}. 

Since $\beta \leq 0$ and $E_P \subset P\subset \bar P$ with $\mu(E_p) \geq \frac{1}{2}\mu(P)\geq c_\nu \mu(\bar P)$, where $c_\nu$ is a constant only dependent on the doubling property of $\mu$ and constants of the dyadic system, we deduce that
$$ \left( \aver{\bar P} u\, d\mu \right)^{\beta} \leq c_\nu^{-\beta} \left( \aver{E_P} u\, d\mu \right)^{\beta}.$$

Then let us define two other weights $\varpi$ and $\rho$
\begin{equation} \varpi := \sigma v^{\frac{p'}{q_0'}}  \qquad \textrm{and} \qquad \rho:=\omega u^{\frac{p}{p_0}}. \label{eq:df}
\end{equation}
Since $u=v^{1-r'}$, an easy computation yields
\begin{equation} u^{-\beta} \varpi^{1/p'} \rho^{1/p}= \sigma^{1/p'} \omega^{1/p} = 1. \label{eq:poids} \end{equation}
By H\"older's inequality with $\gamma:=\frac{1}{1-\beta} \in[0,1]$ and the relation
$$ 1 = \frac{\gamma}{p} + \frac{\gamma}{p'} + (1-\gamma),$$
we have
\begin{equation} \mu(E_P) = \int_{E_P} \left(u^{-\beta} \varpi^{1/p'} \rho^{1/p} \right)^\gamma d\mu \leq u(E_P)^{-\beta \gamma} \varpi(E_P)^{\gamma/p'} \rho(E_P)^{\gamma/{p}}. \label{eq:mso}
\end{equation}
Hence
$$ \left( \aver{E_P} u\, d\mu \right)^{\beta} \mu(E_P) = u(E_P)^{\beta} \mu(E_P)^{1-\beta} \leq \varpi(E_P)^{1/p'} \rho(E_P)^{1/{p}}.$$
So coming back to \eqref{eq:111} we then deduce that
\begin{align*} 
\left| \langle S(f), g\rangle \right| & \lesssim [\omega]^\alpha \sum_{k=1}^K \sum_{P\in \mcS_k} 
\left(\frac{1}{u(\bar P)} \int_{\bar P} |u^{-1/p_0} f|^{p_0} \, ud\mu \right)^{1/p_0} \\
 &  \qquad \qquad \left(\frac{1}{v(\bar P)} \int_{\bar P} |v^{-1/q_0'} g|^{q_0'} \, vd\mu \right)^{1/q_0'} \varpi(E_P)^{1/p'} \rho(E_P)^{1/{p}}.
 \end{align*}
With the dyadic weighted maximal function (see Lemma \ref{lem:max} for its definition) and since $E_P\subset P\subset \bar P$, we deduce that
\begin{align*}
 \left| \langle S(f), g\rangle \right| & \lesssim [\omega]^\alpha \sum_{k=1}^K \sum_{P\in \mcS_k} 
\inf_{E_P} \calM_{u}^{\mcD^k} \left[|u^{-1/p_0} f|^{p_0}\right]^{1/p_0} \inf_{E_P} \calM_{v}^{\mcD^k} \left[ |v^{-1/q_0'} g|^{q_0'} \right]^{1/q_0'} \varpi(E_P)^{1/p'} \rho(E_P)^{1/{p}} \\
 & \lesssim [\omega]^\alpha \sum_{k=1}^K \sum_{P\in \mcS_k} 
\left(\int_{E_P} \calM_{u}^{\mcD^k} \left[|u^{-1/p_0} f|^{p_0}\right]^{p/p_0} \, \rho d\mu\right)^{1/p} \left(\int_{E_P} \calM_{v}^{\mcD^k} \left[ |v^{-1/q_0'} g|^{q_0'} \right]^{p'/q_0'} \, \varpi d\mu \right)^{1/p'}. 
\end{align*}
By H\"older's inequality and using the disjointness of the collection $(E_P)_{P\in \mcS_k}$, one gets
\begin{align*}
 \left| \langle S(f), g\rangle \right| & \lesssim [\omega]^\alpha  
\sum_{k=1}^K \left(\int \calM_{u}^{\mcD^k} \left[|u^{-1/p_0} f|^{p_0}\right]^{p/p_0} \, \rho d\mu\right)^{1/p} \left(\int \calM_{v}^{\mcD^k} \left[ |v^{-1/q_0'} g|^{q_0'} \right]^{p'/q_0'} \, \varpi d\mu \right)^{1/p'}. 
\end{align*}
Since $p\in(p_0,q_0)$, the dyadic maximal function $\calM_{u}^{\mcD^k}$ is $L^{p/p_0}(ud\mu)$ bounded (uniformly in the weight $u$, see Lemma \ref{lem:max}) and it is similar for the weight $v$, hence
\begin{align*}
 \left| \langle S(f), g\rangle \right| & \lesssim  [\omega]^\alpha  
\left(\int |u^{-1/p_0} f|^{p} \, \rho d\mu\right)^{1/p} \left(\int  |v^{-1/q_0'} g|^{p'} \, \varpi d\mu \right)^{1/p'}. 
\end{align*}
Due to the definition \eqref{eq:df} of $\rho,\varpi$, we conclude to
\begin{align*}
 \left| \langle S(f), g\rangle \right| & \lesssim  [\omega]^\alpha  
\left(\int |f|^{p} \, \omega d\mu\right)^{1/p} \left(\int  |g|^{p'} \, \sigma d\mu \right)^{1/p'},
\end{align*}
which corresponds to \eqref{eq:amontrer}.

\medskip
\noindent
{\bf Step 2: } Case $p \leq \mathfrak{p}$ (i.e. $\beta\geq 0$). \\

In this situation, \eqref{eq:eq13} still holds and due to the choice of $\delta$, it yields (instead of \eqref{eq:111})
\begin{align} \left| \langle S(f), g\rangle \right| 
 & \lesssim [\omega]^\alpha \sum_{k=1}^K \sum_{P\in \mcS_k} 
\left(\frac{1}{u(\bar P)} \int_{\bar P} |u^{-1/p_0} f|^{p_0} \, ud\mu \right)^{1/p_0} \nonumber \\
& \qquad \qquad \left(\frac{1}{v(\bar P)} \int_{\bar P} |v^{-1/q_0'} g|^{q_0'} \, vd\mu \right)^{1/q_0'}  \left( \aver{\bar P} v\, d\mu \right)^{\bar \beta}  \mu(E_P), \label{eq:222}
\end{align}
with 
$$ \bar \beta:= \frac{1}{q_0'}-\frac{1}{\delta}=\frac{1}{q_0'}-\frac{1}{p_0(r-1)} = -(r-1) \beta.$$
In particular, since we are in the situation $\beta \geq 0$, we know that $\bar \beta \leq 0$. We can then reproduce a similar reasoning as in the first step, using the inequality
$$ \left( \aver{\bar P} v\, d\mu \right)^{\bar \beta} \lesssim \left( \aver{E_P} v\, d\mu \right)^{\bar \beta}.$$
We use the same weights $\varpi$ and $\rho$, as defined in \eqref{eq:df}, and the exact same computations allow us to conclude, since by definition $u=v^{1-r'}$ which implies
$$ u^{-\beta} = v^{-\beta(1-r')} = v^{-\bar \beta}.$$

\end{proof}

\section{Sharpness of the weighted estimates for the `sparse operators'}

We are going to show that the exponents we obtain previously are sharp for sparse operators. We do so only for dimension $n = 1$, since higher dimensional cases follow through minor modifications.

\medskip
So let us consider the Euclidean space ${\mathbb R}$, equipped with its natural metric and measure.
We first state some easy estimates on specific weights. For $p > 1$ then the weight $w_\alpha: x \mapsto |x|^{\alpha}$ belongs to $A_p$ if and only if  $- 1 < \alpha < p -
1$. One has 
$$[w_{-1+\varepsilon}]_{A_p} \sim \varepsilon^{-1} {\text{ and }} [w_{p-1- \varepsilon}]_{A_p} \sim
\varepsilon^{-(p - 1)}$$ as $\varepsilon \rightarrow 0$.

On the other hand, when $s > 1$ then $w_{- \frac{1}{s} + \varepsilon}$ is
critical for $RH_s$. When $\varepsilon \rightarrow 0$, then 
$$\left[ w_{- \frac{1}{s} + \varepsilon} \right]_{RH_s} \sim
\varepsilon^{-\frac1{s}}.$$

\medskip
Having these sharp estimates, we are now going to prove the optimality of Proposition \ref{prop:sparse}. Considering the particular sparse collection $\mathcal{S}$ of those dyadic intervals contained
in $[0, 1]$ that contain $0$, namely $\mathcal{S}= \{ I_n := [0 ,2^{- n}] : n
\in \mathbb{N} \}$.
Then ${\mathcal S}$ is a sparse collection.
We consider sharpness in the inequality
\begin{equation}\label{sharpsum} 
\sum_{I \in {\mathcal S}} | I | \langle | f |^{p_0} \rangle^{1 /
   p_0}_I \langle | g |^{q'_0} \rangle^{1 / q'_0}_I \lesssim \Phi ([\omega]_{p_0,
   q_0, p}) \| f \|_{L^p_\omega} \| g \|_{L^{p'}_\sigma}, 
\end{equation}
where $1\leq p_0<2<q_0\leq \infty$ are fixed and where to simplify the notation we denote by 
$\langle \cdot \rangle_I$ the average on the interval $I$.

\medskip

\begin{proposition} 
For $p\in(p_0,q_0)$, there exists functions $f,g$ such that asymptotically as $r\to \infty$, the power function $\Phi (r) = r^{\alpha}$ is the best possible choice, where
$\alpha = \frac{1}{p - p_0}$ if $p \in ( p_0, \mathfrak{p}]
$ and $\alpha = \frac{q_0 - 1}{q_0 - p}$ if $p \in 
[\mathfrak{p}, q_0 )$.
\end{proposition}

Notice that for $q_0=\infty$ the above sum corresponds to the pointwise defined operator \[ S f = \sum_{I \subseteq \mathcal{D} [0, 1], 0 \in I} \langle | f^{p_0} |
   \rangle^{1 / p_0}_I \chi_I\] tested against $g$.

For convenience, we also will use the following notation (introduced in \cite{AM}): for $\omega$ a weight then
$$ [\omega]_{p_0,q_0,p}:=[\omega]_{A_{\frac{p}{p_0}}} [\omega]_{ RH_{\left(\frac{q_0}{p}\right)'}}.$$

\begin{proof} Let $p \in ( p_0, \mathfrak{p}] $. Consider functions $f_{\varepsilon} :=x\mapsto x^{- 1 / p_0 + \varepsilon} \chi_{[0, 1]}$ and $g_{\varepsilon} :=x\mapsto  x^{- 1 / p'_0 +
\varepsilon} \chi_{[0, 1]}$. One calculates for $I_n = [0, 2^{- n}]$ with $n
\geqslant 0$ that
\[ \langle | f_{\varepsilon} |^{p_0} \rangle^{1 / p_0}_{I_n} =  \frac{2^{\frac{n}{p_0}  - n
   \varepsilon}}{\left( p_0 \varepsilon \right)^{1 / p_0}} \sim \varepsilon^{- 1 /
   p_0} 2^{- n \varepsilon} 2^{n / p_0} \]
and
\[ \langle | g_{\varepsilon} |^{q'_0} \rangle^{1 / q'_0}_{I_n} = \frac{2^{\frac{n}{p'_0} - n
   \varepsilon}}{\left( 1 - \frac{q'_0}{p'_0} + q'_0 \varepsilon \right)^{1 /
   q'_0}} \sim 2^{- n \varepsilon} 2^{n / p'_0} \]
by noticing that $q'_0 / p'_0 < 1$.

Hence, we obtain for the left hand side of (\ref{sharpsum})
\[ \varepsilon^{- 1 / p_0}
   \sum^{\infty}_{n = 0} 2^{- 2 n \varepsilon} = \varepsilon^{- 1 / p_0}
   \frac{1}{1 - \left( \frac{1}{4} \right)^{\varepsilon}} \sim \varepsilon^{-
   1 / p_0} \varepsilon^{- 1} . \]

Choose the weight $\omega_{\varepsilon} = w_{\frac{p}{p_0} - 1 - \varepsilon}:=x\mapsto x^{\frac{p}{p_0} - 1 - \varepsilon}$, which is critical
for $A_{p / p_0}$ with \ $[\omega_{\varepsilon}]_{A_{\frac{p}{p_0}}} \sim
\varepsilon^{-(\frac{p}{p_0} - 1)}$ as $\varepsilon \rightarrow 0$. We also
notice that $\omega_{\varepsilon}$ is a power weight of positive exponent and therefore
$[\omega_{\varepsilon}]_{RH_{\left( \frac{q_0}{p} \right)'}} \sim 1$ as $\varepsilon
\rightarrow 0$. So $[\omega_{\varepsilon}]_{p_0,q_0, p} \sim \varepsilon^{-(\frac{p}{p_0} -
1)}$ and $[\omega_{\varepsilon}]^{\frac1{p-p_0}}_{p_0, q_0, p} \sim \varepsilon^{-\frac1{p_0}}$. 
We calculate
\[ \| f_{\varepsilon} \|_{L^p_{\omega_{\varepsilon}}} = 
%\frac{1}{((p - 1) \varepsilon)^{1 / p}} 
\left(\int^1_0x^{-1+(p-1)\varepsilon}dx\right)^{1/p}
\sim
   \varepsilon^{- 1 / p} . \]
With $\sigma_{\varepsilon} = \omega_{\varepsilon}^{1 - p'}$ we calculate
\[ \| g_{\varepsilon} \|_{L^{p'}_{\sigma_{\varepsilon}}} = 
%\left( \int^1_0 x^{- p' / p'_0 + p' \varepsilon} x^{(p / p_0 - 1 - \varepsilon) (1 - p')} d x \right)^{1 / p'} = 
\left( \int^1_0 x^{(2 p' - 1) \varepsilon - 1} d x \right)^{1 / p'} \sim
   \varepsilon^{- 1 / p'} . \]

Gathering the information gives $\varepsilon^{- 1 / p} \varepsilon^{- 1
/ p'} \varepsilon^{- 1 / p_0}$ on the right hand side and $\varepsilon^{- 1}
\varepsilon^{- 1 / p_0}$ on the left, showing that the choice of $\Phi$ cannot
be improved for this range of $p$.

\medskip
Let now $p \in [\mathfrak{p}, q_0)$. To treat this
range, we apply what we have found before to the modified exponents $1 \leq q'_0 < 2 < p'_0
\leq \infty$. We have seen examples of sharpness for the sum

\[ \sum_{I \in {\mathcal S}} | I | \langle | f |^{q'_0} \rangle^{1
   / q'_0}_I \langle | g |^{p_0} \rangle^{1 / p_0}_I \sim [\omega]_{q'_0, p'_0,
   s}^{\frac{1}{s - q'_0}} \| f \|_{L^s_\omega} \| g \|_{L^{s'}_\sigma} \]
when $q'_0 \leq s \leq \mathfrak{p} (q'_0, p'_0)$. Indeed, with $f_{\varepsilon} :=x\mapsto x^{- 1 / q'_0 + \varepsilon} \chi_{[0, 1]}$, $g_{\varepsilon} :=x\mapsto x^{- 1 / q_0 +
\varepsilon} \chi_{[0, 1]}$ and $\omega_{\varepsilon} :=x\mapsto | x |^{s / q'_0 - 1 - \varepsilon}$ we
obtain for the left hand side $\sim \varepsilon^{- 1} \varepsilon^{- 1 /
q'_0}$ and $\| f_{\varepsilon} \|_{L^s_{\omega_{\varepsilon}}} \sim \varepsilon^{- 1 / s}$ and $\| g_{\varepsilon} \|_{L^{s'}_\sigma} \sim \varepsilon^{- 1 / s'}$. Now observe that $\big[\mathfrak{p} (q'_0,
p'_0)\big]' =\mathfrak{p} (p_0, q_0)$. Observe also that therefore $\mathfrak{p}
(p_0, q_0) \leqslant s' \leqslant q_0$. Using this for $s' = p$, it remains to
calculate $[\sigma_{\varepsilon}]^{\frac{q_0 - 1}{q_0 - p}}_{p_0, q_0, p}$ where $\sigma_{\varepsilon} =
\omega_{\varepsilon}^{1 - p}$. $$\sigma_{\varepsilon}(x) = | x |^{(p' / q_0' - 1 - \varepsilon) (1 - p)} = | x |^{-
1 / (q_0 / p)' + (p - 1) \varepsilon}.$$ This weight is of negative exponent
and critical for $RH_{\left( \frac{q_0}{p} \right)'}$ with
$[\sigma_{\varepsilon}]_{p_0, q_0, p} \sim \varepsilon^{- 1 / \left( \frac{q_0}{p}
\right)'}$. Therefore $[\sigma_{\varepsilon}]_{p_0, q_0, p}^{\frac{q_0 - 1}{q_0 - p}} \sim
\varepsilon^{- \frac{1}{q'_0}}$. Gathering the information, we obtain $\sim
\varepsilon^{- 1} \varepsilon^{- 1 / q'_0}$ for the left hand side and $\sim
\varepsilon^{- 1/q'_0} \varepsilon^{- 1/p} \varepsilon^{-1/p'}$ when using 
$\Phi (r) = r^{\alpha}$, showing that the estimate
cannot be improved.
\end{proof}

\end{document}